\documentclass[10pt]{amsart}

\title{Marked length rigidity for Fuchsian buildings}

\author{David Constantine}
\address{
Wesleyan University \\
Mathematics and Computer Science Department \\
Middletown, CT 06459}

\author{Jean-Fran\c{c}ois Lafont}
\address{Department of Mathematics\\
                 Ohio State University\\
                 Columbus, Ohio 43210}

\date{\today}

\newtheorem{thm}{Theorem}
\newtheorem*{theorem}{Main Theorem}
\newtheorem{lem}[thm]{Lemma}
\newtheorem{cor}[thm]{Corollary}
\newtheorem{prop}[thm]{Proposition}

\theoremstyle{definition}
\newtheorem*{rem}{Remark}
\newtheorem{exmp}[thm]{Example}

\newtheorem{defn}[thm]{Definition}

\numberwithin{equation}{section}
\numberwithin{thm}{section}

\def\Pb{\ifmmode{\Bbb P}\else{$\Bbb P$}\fi}
\def\Z{\ifmmode{\Bbb Z}\else{$\Bbb Z$}\fi}
\def\Q{\ifmmode{\Bbb Q}\else{$\Bbb Q$}\fi}
\def\C{\ifmmode{\Bbb C}\else{$\Bbb C$}\fi}
\def\R{\ifmmode{\Bbb R}\else{$\Bbb R$}\fi}
\def\H{\ifmmode{\Bbb H}\else{$\Bbb H$}\fi}

\def\Isom{\operatorname{Isom}}

\def\aut{\operatorname{Aut}}

\newcommand{\transv}{\mathrel{\text{\tpitchfork}}}
\makeatletter
\newcommand{\tpitchfork}{%
  \vbox{
    \baselineskip\z@skip
    \lineskip-.52ex
    \lineskiplimit\maxdimen
    \m@th
    \ialign{##\crcr\hidewidth\smash{$-$}\hidewidth\crcr$\pitchfork$\crcr}
  }%
}
\makeatother

\newcommand{\ntransv}{\mathrel{\text{\ntpitchfork}}}
\makeatletter
\newcommand{\ntpitchfork}{%
\raise 1ex
  \vbox{
    \baselineskip\z@skip
    \lineskip-1.2ex
    \lineskiplimit\maxdimen
    \m@th
    \ialign{##\crcr\hidewidth\smash{$\diagup$}\hidewidth\crcr$\pitchfork$\crcr\smash{$-$}\crcr}
  }%
}
\makeatother

\usepackage{graphicx} 
\usepackage{mathtools}
\usepackage{graphics}
\usepackage{amsmath,amssymb,amsthm,amsfonts,amscd,flafter,epsf,bm}
\usepackage{epsfig}
\usepackage{psfrag}
\usepackage{color}
\usepackage{mathrsfs}
\input xy
\xyoption{all}
\usepackage{tikz}

\begin{document}

\begin{abstract}

We consider finite $2$-complexes $X$ that arise as quotients of Fuchsian buildings by subgroups of the combinatorial
automorphism group, which we assume act freely and cocompactly. We show that locally CAT(-1) metrics on $X$ which are piecewise
hyperbolic, and satisfy a natural non-singularity condition at vertices are marked length spectrum rigid within certain classes
of negatively curved, piecewise Riemannian metrics on $X$. As a key step in our proof, we show that the marked length spectrum function for such metrics determines the volume of $X$.
\end{abstract}

\maketitle

\setcounter{secnumdepth}{1}

\setcounter{section}{0}

\section{\bf Introduction}

One of the central results in hyperbolic geometry is Mostow's rigidity theorem, which states that for closed hyperbolic manifolds of dimension $\geq 3$, isomorphism of fundamental groups implies isometry. Moving away from the constant curvature case, one must impose some additional constraints on the isomorphism of fundamental groups if one hopes to conclude it is realized by an isometry. On any closed negatively curved manifold $M$, each free homotopy class of loops contains a unique geodesic representative. This gives a well-defined class function $MLS: \pi_1(M)\rightarrow \mathbb R^+$, called the {\it marked length spectrum function}. Given a pair of negatively curved manifolds $M_0, M_1$, we say they have {\it the same} marked length spectrum if there is an isomorphism $\phi: \pi_1(M_0) \rightarrow \pi_1(M_1)$ with the property that $MLS_1\circ \phi = MLS_0$. The marked length spectrum conjecture predicts that closed negatively curved manifolds with the same marked length spectrum must be isometric (and that the  isomorphism of fundamental groups is induced by an isometry). In full generality, the conjecture is only known to hold for closed surfaces, which was independently established by Croke \cite{croke} and Otal \cite{otal}. In the special case where one of the Riemannian metrics is locally symmetric, the conjecture was established by Hamenst\"adt \cite{ham} (see also Dal'bo and Kim \cite{dalbo-kim} for analogous results in the higher rank case).  

Of course, it is possible to formulate the marked length spectrum conjecture for other classes of geodesic spaces -- for example, compact locally CAT(-1) spaces. Still in the realm of surfaces, Hersonsky and Paulin \cite{hersonsky-paulin} extended the result  to some singular metrics on surfaces, while Bankovi\'c and Leininger \cite{bankovic-leininger} and Constantine \cite{constantine} give extensions to the case of non-positively curved metrics. Moving away from the surface case, the conjecture was verified independently by Alperin and Bass \cite{alperin-bass} and by Culler and Morgan \cite{culler-morgan} in the special case of locally CAT(-1) spaces whose universal covers are metric trees. This was recently extended by the authors to the context of compact geodesic spaces of topological (Lebesgue) dimension one, see \cite{CL}. 

In this paper, we are interested in the marked length spectrum conjecture for compact quotients of {\it Fuchsian buildings}, a class of polygonal $2$-complexes supporting locally CAT(-1) metrics. Fixing such a quotient $X$, we can then look at various families of locally negatively curved metrics on $X$. The 
metrics we consider are {\it piecewise Riemannian}: each polygon in the complex is equipped with a Riemannian metric with geodesic boundary edges. They are also assumed to be locally negatively curved, which means that the metrics satisfy Gromov's ``large link condition'' at all the vertices. We consider three classes of such metrics: those whose curvatures are everywhere bounded above by -1, those whose curvature is everywhere hyperbolic, and those whose curvatures are everywhere within the interval $[-1, 0)$. The space of such metrics will be denoted $\mathcal M_{\leq} (X)$, $\mathcal M_{\equiv}(X)$, and $\mathcal M_{\geq}(X)$ respectively. Note that the family of piecewise hyperbolic metrics $\mathcal M_{\equiv}(X)$ are precisely the metrics lying in the intersection $\mathcal M_{\leq}(X) \cap \mathcal M_{\geq}(X)$. Furthermore, all three of these classes of metrics lie within the space $\mathcal M_{neg}(X)$, consisting of all (locally) negatively curved, piecewise Riemannian metrics on $X$. Finally, if we impose some further regularity conditions on the vertices, we obtain subclasses of metrics $\mathcal M_{\leq}^v(X)$, $\mathcal M_{\equiv}^v(X)$, $\mathcal M_{\geq}^v(X)$, and $\mathcal M_{neg}^v(X)$. We refer our reader to Section \ref{background} for 
further background on Fuchsian buildings, including precise definitions for these classes of metrics -- let us just mention that, amongst these, the most ``regular'' metrics are those lying in the class $\mathcal M_{\equiv}^v(X)$, which forms an analogue of Teichm\"uller space for $X$.

\vskip 10pt

\begin{theorem}
Let $X$ be a quotient of a Fuchsian building $\tilde X$ by a subgroup $\Gamma \leq \aut (\tilde X)$ of the combinatorial automorphism group $\aut (\tilde X)$ which acts freely and cocompactly. Consider a pair of negatively curved metrics $g_0, g_1$ on $X$, where $g_0$ is in $\mathcal M_{\equiv}^v(X)$, and $g_1$ is in $\mathcal M_{\geq}^v(X)$.  Then $(X, g_0)$ and $(X, g_1)$ have the same marked length spectrum if and only if they are isometric.
\end{theorem}

In the process of establishing the {\bf Main Theorem}, we also obtain a number of auxiliary results which may be of some independent interest. Let us briefly mention a few of these. Throughout the rest of this section, $X$ will denote a quotient of a Fuchsian building $\tilde X$ by a subgroup $\Gamma \leq \aut (\tilde X)$ which acts freely and cocompactly.

The first step is to obtain marked length spectrum rigidity for certain pairs of metrics in $\mathcal M_{\leq}(X)$.

\begin{thm}[MLS rigidity -- special case]\label{thm:MLS-rigidity-special-case}
Let $g_0 , g_1$ be any two metrics in $\mathcal M_\equiv ^v$ and $\mathcal M_{\leq}(X)$ respectively. Then $(X, g_0)$ and $(X, g_1)$ have the same marked length spectrum if and only if they are isometric.
\end{thm}

This result is established in Section 3, and is based on an argument outlined to us by an anonymous referee. Next, we study the volume functional on the space of metrics. We note that the volume is constant on the subspace $\mathcal M_{\equiv}^v(X)$, and in Section 4, we show the following rigidity result:

\begin{thm}[Minimizing the volume]\label{thm:minimizing-the-volume}
Let $g_0$ be a metric in $\mathcal M_{\equiv}^v(X)$, and $g_1$ an arbitrary metric in $\mathcal M_{\geq}^v(X)$. If $Vol(X, g_1) \leq Vol(X, g_0)$, then $g_1$ must lie within $\mathcal M_{\equiv}^v(X)$ (and the inequality is actually an equality).
\end{thm}

Finally, the last (and hardest) step in the proof is a general result relating the marked length spectrum and the volume. We show:

\begin{thm}[MLS determines volume]\label{thm:MLS-determines-volume}
Let $g_0, g_1$ be an arbitrary pair of metrics in $\mathcal M_{neg}(X)$. If $MLS_0 \leq MLS_1$, then $Vol(X, g_0) \leq Vol(X, g_1)$.
\end{thm}

The analogous result for negatively curved metrics on a closed surface is due to Croke and Dairbekov \cite{croke-dairbekov}, who also established a version for conformal metrics on negatively curved manifolds (see also some related work by Fana\"i \cite{Fa} and by Z. Sun \cite{sun}). Our proof of Theorem \ref{thm:MLS-determines-volume} roughly follows the approach in \cite{croke-dairbekov}.  After setting up the preliminaries in Section \ref{sec:flows and currents}, we introduce in Sections \ref{sec:transverse} and \ref{sec:pairings} a new notion of intersection pairing, a central tool in Otal's and Croke and Dairbekov's work on the marked length spectrum. 
Our pairing relies only on the combinatorics of the building, and thus is metric independent. However, we show in Section \ref{sec:compute}
that this combinatorial intersection pairing, when applied to geometrically defined currents, still captures some of the geometry of the underlying metric. In Sections \ref{sec:geom lemma} and \ref{sec:cty} we show a weak form of continuity for the combinatorial intersection pairing, 
evaluated along certain specific sequences of currents.  These properties of the combinatorial intersection pairing are
then used to prove Theorem \ref{thm:MLS-determines-volume} in Section \ref{sec:MLS-determines-volume}.

Finally, using these three theorems, the proof of the {\bf Main Theorem} is now straightforward.

\begin{proof}[Proof of Main Theorem]
Let $g_0$ be a metric in $\mathcal M_{\equiv}^v(X)$, and $g_1$ a metric in $\mathcal M^v_{\geq}(X)$. If $MLS_0 \equiv MLS_1$, then by Theorem \ref{thm:MLS-determines-volume}, we see that $Vol(g_1)=Vol(g_0)$. So Theorem \ref{thm:minimizing-the-volume} forces $g_1$ to lie in the space $\mathcal M_{\equiv}^v(X)$. 
Since they have the same marked length spectrum, Theorem \ref{thm:MLS-rigidity-special-case} now allows us to conclude that $(X, g_0)$ is isometric to $(X, g_1)$, completing the proof.
\end{proof}

These results provide partial evidence towards the general marked length spectrum conjecture for these compact quotients of Fuchsian buildings, which we expect to hold for any pairs of metrics in $\mathcal M_{neg}(X)$. We should mention that rigidity theorems for such quotients $X$ are often difficult to prove. For instance combinatorial (Mostow) rigidity was established by Xiangdong Xie \cite{X} (building on previous work of Bourdon \cite{Bou2}). Quasi-isometric rigidity was also established by Xie \cite{X}, generalizing earlier work of Bourdon and Pajot \cite{BP}. Superrigidity with targets in the isometry group of $\tilde X$ was established by Daskalopoulos, Mese, and Vdovina \cite{DMV}. Finally, in the context of volume entropy, recent work of Ledrappier and Lim \cite{LL} leaves us uncertain as to which metrics in $\mathcal M_{\equiv}(X)$ minimize the volume growth entropy (they show that the ``obvious'' candidate for a minimizer is actually not a minimizer).

\subsection{Acknowledgements} The first named author would like to thank Ohio State for hosting him for several visits during which a portion of this work was completed. The second author was partially supported by the NSF, under grants DMS-1207782, DMS-1510640.
The authors would also like to thank the anonymous referee for informing us of Bourdon's work, and suggesting the proof
of Theorem \ref{thm:MLS-rigidity-special-case} which is given in Section 3. Our original argument for this result was 
considerably more involved, and also included some unnecessary assumptions. The authors are also
indebted to the referee for asking probing questions, which led the authors to discover a serious gap in an earlier
draft of the paper. Finally, we would also like to thank Marc Bourdon and Alina Vdovina for helpful comments.

%%%%%%%%%%%%%%%%%%%%%%%%%%%%%
%%%%%%%%%%%%%%%%%%%%%%%%%%%%%

\section{Background material}\label{background}

%
%%
%%%%%%%%%%%%%%%%

\subsection{Fuchsian buildings.}

We start by summarizing basic notation and conventions on Fuchsian buildings, which were first introduced by Bourdon \cite{Bou3}. These are $2$-dimensional polyhedral complexes which satisfy a number of axioms. First, one starts with a compact convex hyperbolic polygon $R\subset \mathbb H^2$, with each angle of the form $\pi/m_i$ for some $m_i$ associated to the vertex ($m_i\in \mathbb N, m_i\geq 2$). Reflection in the geodesics extending the sides of $R$ generate a Coxeter group $W$, and the orbit of $R$ under $W$ gives a tessellation of $\mathbb H^2$. Cyclically labeling the edges of $R$ by the integers $\{1\}, \ldots, \{k\}$ (so that the vertex between the edges labelled $i$ and $i+1$ has angle $\pi/m_i$), one can apply the $W$ action to obtain a $W$-invariant labeling of the tessellation of $\mathbb H^2$; this edge labeled polyhedral $2$-complex will be denoted $A_R$, and called the {\it model apartment}.

A polygonal $2$-complex $\tilde X$ is called a $2$-dimensional hyperbolic building if it contains an edge labeling by the integers $\{1, \ldots , k\}$, along with a distinguished collection of subcomplexes $\mathcal A$ called the {\it apartments}. The individual polygons in $\tilde X$ will be called {\it chambers}. The complex is required to have the following properties:
\begin{itemize}
	\item each apartment $A\in \mathcal A$ is isomorphic, as an edge labeled polygonal complex, to the model apartment $A_R$,
	\item given any two chambers in $\tilde X$, one can find an apartment $A\in \mathcal A$ which contains the two chambers, and
	\item given any two apartments $A_1, A_2\in \mathcal A$ that share a chamber, there is an isomorphism of labeled 
$2$-complexes $\varphi: A_1\rightarrow A_2$ that fixes $A_1\cap A_2$. 
\end{itemize}
If in addition each edge labeled $i$ has a fixed number $q_i$ of incident polygons, then $\tilde X$ is called a {\it Fuchsian building}. The group $\aut(\tilde X)$ will denote the group of combinatorial (label-preserving) automorphisms of the Fuchsian building $\tilde X$.

Throughout this paper we make the standing assumption that $\tilde X$ is \emph{thick}, i.e. that every edge is contained in at least three chambers. Thus, the overall geometry of the building $\tilde X$ will involve an interplay between the geometry of the apartments, and the combinatorics of the branching along the edges.

Note that making each polygon in $\tilde X$ isometric to $R$ via the label-preserving map produces a CAT(-1) metric on $\tilde X$. However, a given polygonal $2$-complex might have several metrizations as a Fuchsian building: these correspond to varying the hyperbolic metric on $R$ while preserving the angles at the vertices. Any such variation induces a new CAT(-1) metric on $\tilde X$. The hyperbolic polygon $R$ is called {\it normal} if it has an inscribed circle that touches all its sides -- fixing the angles of a polygon to be $\{\pi/m_1, \ldots ,\pi/m_k\}$, there is a unique normal hyperbolic polygon with those given vertex angles.  We will call the quantity $\pi/m_i$ the {\it combinatorial angle} associated to the corresponding vertex. A Fuchsian building will be called {\it normal} if all metric angles are equal to the corresponding combinatorial angles {\it and} the metric on each chamber is normal. We can now state Xiangdong Xie's version of Mostow rigidity for Fuchsian buildings (see \cite{X}):

\begin{thm}[Xie]\label{xie}
Let $\tilde X_1, \tilde X_2$ be a pair of Fuchsian buildings, and let $\Gamma_i\leq \Isom(\tilde X_i)$ be a uniform lattice. Assume that we have an isomorphism $\phi: \Gamma_1\rightarrow \Gamma_2$. Then there is a $\phi$-equivariant homeomorphism $\Phi: \tilde X_1\rightarrow \tilde X_2$. Moreover, if both buildings are normal, then one can choose $\Phi$ to be a $\phi$-equivariant isometry.
\end{thm}

Another notion that will reveal itself useful is the following: inside $\tilde X$, we have a collection of {\it walls}, which are defined as follows. First, recall that each apartment in the building is (combinatorially) modeled on a $W$-invariant polygonal tessellation of $\mathbb H^2$. The geodesics extending the various sides of the polygons give a $W$-invariant collection of geodesics in $\mathbb H^2$, which are also a collection of combinatorial paths in the tessellation. This gives a distinguished collection of combinatorial paths in the model apartment $A_R$ -- its walls. Via the identification of apartments $A \in \mathcal A$ in $\tilde X$ with the model apartment $A_R$, we obtain the notion of wall in an apartment of $\tilde X$. Note that every edge in $\tilde X$ is contained in many different walls of $\tilde X$.

%
%%
%%%%%%%%%%%%%%%%

\subsection{Structure of vertex links}

For a Fuchsian building, the combinatorial axioms force some additional structure on the links of vertices: these graphs must be (thick) {\it generalized $m$-gons}  (see for instance \cite[Prop. 4.9 and 4.44]{brown}). Work of Feit and Higman \cite{FH} then implies that each $m_i$ must lie in the set $\{2, 3, 4, 6, 8\}$. Viewed as a combinatorial graph, a generalized $m$-gon has diameter $m$ and girth $2m$. Moreover, taking the collection of cycles of length $2m$ within the graph to be the set of apartments, such a graph has the structure of a (thick) spherical building (based on the action of the dihedral group $D_{2m}$ of order $2m$ acting on $S^1$).

For instance, when $m=2$, a generalized $2$-gon is just a complete bipartite graph $K_{p,q}$. When $m=3$, generalized $3$-gons correspond to the incidence structure on finite projective planes (whose classification is a notorious open problem). When $m>3$, examples are harder to find. An extensive discussion of generalized $4$-gons can be found in the book \cite{payne-thas}. For generalized $6$-gons and $8$-gons, the only known examples arise from certain incidence structures associated to some of the finite groups of Lie type (see e.g. \cite{van-Maldeghem}).

Note that, at a given vertex $v$, the edges incident to $v$ always have one of two possible (consecutive) labels. On the level of the link, this means that $lk(v)$ comes equipped with an induced $2$-coloring of the vertices by the integers $i, i+1$. Since all edges with a given label $i$ have $q_i$ incident chambers, this means that the vertices in $lk(v)$ colored $i,i+1$ have degrees $q_i, q_{i+1}$ respectively. In the case of generalized $2$-gons, the vertex $2$-coloring is the one defining the complete bipartite graph structure. For a generalized $3$-gon, the identification of the graph with the incidence structure of a finite projective plane $\mathcal P$ provides the $2$-coloring: the colors determine whether a vertex in the graph corresponds to a point or to a line in $\mathcal P$.

Split the vertex set into $\mathcal V_i, \mathcal V_{i+1}$, the set of vertices with label $i, i+1$ respectively. From the bipartite nature of the graph, the number of edges in the graph satisfies $|\mathcal E| = q_i |\mathcal V_i| = q_{i+1}|\mathcal V_{i+1}|$. Given an edge $e\in \mathcal E$, we now count the number of apartments (i.e. $2m$-cycles) passing through $e$. In a generalized $m$-gon, any path of length $m+1$ is contained in a unique apartment (see, e.g. \cite[Prop. 7.13]{Weiss-structure}). Thus, to count the number of apartments through $e$, it is enough to count the number of ways to extend $e$ to a path of length $m+1$. The number of branches we can take at each vertex alternates between a $q_i$ and a $q_{i+1}$. So if $m$ is even, we obtain that the number of edges is $N:= q_i^{m/2}q_{i+1}^{m/2}$.

If $m=3$ is odd, then we note that $q_i = q_{i+1}$. Indeed, opposite vertices in one of the $6$-cycles have labels $q_i$ and $q_{i+1}$. But for each vertex in $lk(v)$ (which corresponds to an edge in the original building) the valence corresponds to the number of chambers which share that edge. Since the branching in the ambient building occurs along walls, for two opposite vertices in an apartment in the link, the valence must be the same. So in this case, the number of apartments through an edge is $N:= q_i^3 = q_{i+1}^3$.

%
%%
%%%%%%%%%%%%%%%%%%

\subsection{Spaces of metrics}

Now consider a compact quotient $ X = \tilde X/\Gamma$ of a Fuchsian building, where $\Gamma\subset Aut(\tilde X)$ is a lattice in the group of combinatorial automorphisms of $\tilde X$. On the quotient space $X$, we will consider metrics which are {\it piecewise Riemannian}, i.e. whose restriction to each chamber of $X$ is a Riemannian metric, such that all the sides of the chamber are geodesics. Moreover, we will restrict to metrics which are locally negatively curved -- and thus will require the metrics on each chamber to have sectional curvature $<0$. We will denote this class of metrics by $\mathcal{M} _{neg}$. If we instead require each chamber to be hyperbolic (i.e. to have curvature $\equiv -1$), then we obtain the space $\mathcal{M}_{\equiv}$. Similarly, we can require each chamber to have curvature $\leq -1$, or curvature in the interval $[-1, 0)$. These give rise to the corresponding spaces $\mathcal{M}_{\leq}$ or $\mathcal{M}_{\geq}$, respectively. Clearly, we have a proper inclusion $\mathcal M_\leq \cup \mathcal{M}_{\equiv} \cup \mathcal M_\geq \subset \mathcal{M}_{neg}$, as well as the equality $\mathcal{M}_\equiv = \mathcal M_\leq \cap \mathcal M_\geq$. Notice that, for all of these classes of metrics, the negative curvature property imposes some constraints on the metric near the vertices of $X$: they must always satisfy Gromov's ``large link condition" (see discussion below).

In order to obtain a true analogue of hyperbolic metrics on $X$, one needs to impose some additional regularity condition. To illustrate this, consider the case of piecewise hyperbolic metrics on ordinary surfaces. One can pullback a hyperbolic metric on a surface $\Sigma _2$ of genus two via a degree two map $\Sigma_4\rightarrow \Sigma_2$ ramified over a pair of points. The resulting metric on the surface $\Sigma_4$ of genus four is piecewise hyperbolic, but has two singular points with cone angle $=4\pi$, so in particular is {\it not} hyperbolic. By analogy, an analogue of a constant curvature metric on $X$ should have ``as few'' singular points as possible.

Of course, the only possible singularities occur at the vertices of $X$. Given a vertex $\tilde v \in \tilde X$, one has several apartments passing through $\tilde v$, and one can restrict the metric to each of these apartments. The negative curvature condition implies that each of these apartments inherits a (possibly singular) negatively curved metric. This tells us that the sum of the angles around the vertex $\tilde v$ in each apartment is $\geq 2\pi$. We say that the vertex $\tilde v$ is {\it metrically non-singular} if, when restricted to each apartment through $\tilde v$, the sum of the angles at $\tilde v$ is exactly $2\pi$. A metric has {\it non-singular vertices} if every vertex is metrically non-singular. We will denote the subspace of such metrics inside $\mathcal{M}_{neg}$ by $\mathcal{M}_{neg}^{v}$. We can similarly define the subsets $\mathcal{M}_{\leq}^v$, $\mathcal M_\equiv^v$ and $\mathcal M_\geq^v$ inside the spaces $\mathcal{M}_{\leq}, \mathcal M_\equiv , \mathcal M_\geq$ respectively (the superscript $v$ is intended to denote non-singular vertices).

When $X$ is equipped with a piecewise Riemannian metric $g$, each vertex link $lk(v)$ gets an induced metric $d$. Indeed, an edge in $lk(v)$ corresponds to a chamber corner in $X$. Since the chamber $C$ has a Riemannian metric with geodesic sides, the corner has an angle $\theta$ measured in the $g$ metric. The $d$-length of the corresponding edge is defined to be the angle $\theta$. With respect to this metric, the negative curvature condition at $v$ translates to saying that every $2m$-cycle in the generalized $m$-gon $lk(v)$ has total $d$-length $\geq 2\pi$ (Gromov's ``large link'' condition).  The metric $g\in \mathcal{M}_{neg}$ lies in the subclass $ \mathcal{M}_{neg}^{v}$ precisely if for every vertex link $lk(v)$, the metric $d$-length of {\it every} $2m$-cycle is {\it exactly} $2\pi$. Of course, a similar statement holds for $\mathcal{M}_{\leq}^{v}, \mathcal M^v_\equiv , \mathcal M^v _\geq$.  As we will see below (Corollary \ref{metric-angles-equal-comb-angles}), the non-singularity condition on vertices imposes very strong constraints on the vertex angles -- they will always equal the corresponding combinatorial angle.

%
%%
%%%
%%%%
%%%%%%%%%%%%%%%%%%%%%%%%%%%%%%%%%%%%%%%%%%%%%%%%%%%%%%%%%%%%%%%

\section{MLS rigidity for metrics in $\mathcal{M}_{\equiv}^{v}$ }

This section is devoted to proving Theorem \ref{thm:MLS-rigidity-special-case}. The argument we present here was suggested to us by the anonymous referee. We start by reminding the reader of some metric properties of boundaries of CAT(-1) spaces. If $(\tilde X,d)$ is any CAT(-1) space, with boundary at infinity $\partial ^\infty (\tilde X,d)$, the {\it cross-ratio} is a function on $4$-tuples $(\xi, \xi', \eta, \eta')$ of distinct points in $\partial ^\infty (\tilde X,d)$. It is defined by:
$$[\xi \xi' \eta \eta'] := \lim_{(a, a', b, b') \to (\xi, \xi', \eta, \eta')} \text{Exp}\left(\frac{1}{2}\left(d(a, b) + d(a', b') - d(a, b')-d(a',b) \right)\right)$$
and the $4$-tuple $(a, a', b, b')$ converges radially towards $(\xi, \xi', \eta, \eta')$. If $\tilde Y$ is another CAT(-1) space, a topological embedding $\Phi: \partial ^\infty \tilde Y \rightarrow \partial ^\infty \tilde X$ is a {\it M\"obius map} if it respects the cross-ratio, i.e. for all $4$-tuples of distinct points $(\xi, \xi', \eta, \eta')$ in $\partial ^\infty Y$, we have
$$[\Phi(\xi) \Phi(\xi') \Phi(\eta) \Phi(\eta')] = [\xi \xi' \eta \eta'].$$
Note that an isometric embedding of CAT(-1) spaces automatically induces a M\"obius map between boundaries at infinity. As a consequence, for a totally geodesic subspace of a CAT(-1) space, the intrinsic cross-ratio (defined from within the subspace) coincides with the extrinsic cross-ratio (restriction of the cross-ratio from the ambient space).

\begin{proof}[Proof of Theorem \ref{thm:MLS-rigidity-special-case}]
Lifting the metrics $g_0, g_1$ to the universal cover, the identity map lifts to a quasi-isometry $\Phi: (\tilde X, \tilde g_0) \rightarrow (\tilde X, \tilde g_1)$. This induces a map between boundaries at infinity $\partial ^\infty \Phi: \partial ^\infty (\tilde X, \tilde g_0) \rightarrow \partial ^\infty (\tilde X, \tilde g_1)$. Otal showed that, if an isomorphism of fundamental groups preserves the marked length spectrum, then the induced map on the boundaries at infinity is M\"obius (see \cite{otal-2} -- the argument presented there is for negatively curved closed manifolds, but the proof extends verbatim to the $CAT(-1)$ setting).

Now let $\mathcal A$ be the collection of apartments in the building $\tilde X$ (note that this is independent of the choice of metric on $X$). Since $g_0\in \mathcal M_\equiv ^v(X)$, each apartment $A \subset \tilde X$ inherits a piecewise hyperbolic metric, with no singular vertices. So each $(A, \tilde g_0|_A)$ is a totally geodesic subspace of $(\tilde X, \tilde g_0)$, isometric to $\mathbb H^2$. The map $\partial ^\infty \Phi$ sends the circle corresponding to $\partial ^\infty (A, \tilde g_0|_A)$ to the circle in $\partial ^\infty (\tilde X, \tilde g_1)$ corresponding to the totally geodesic subspace  $(A, \tilde g_1|_A)$ (see \cite{X}). Since the map $\partial ^\infty \Phi$ preserves the cross-ratio, work of Bourdon \cite{Bou1} implies that there is an isometric embedding $F_A:(A, \tilde g_0|_A) \rightarrow (\tilde X, \tilde g_1)$ which ``fills-in'' the boundary map. This isometry must have image $(A, \tilde g_1|_A)$, which hence must also be isometric to $\mathbb H^2$. Applying this to every apartment, we see that the metric $g_1$, which was originally assumed to be in $\mathcal M_{\leq}(X)$, must actually lie in the subspace $\mathcal M^v_{\equiv}(X)$.

Finally, we claim that there is an equivariant isometry between $(\tilde X, \tilde g_0)$ and $(\tilde X, \tilde g_1)$. For each apartment $A\in \mathcal A$, we have an isometry $F_A:(A, \tilde g_0|_A) \rightarrow (A, \tilde g_1|_A)$. From Xie's work, the boundary map $\partial ^\infty F_A \equiv \partial ^\infty \Phi |_{\partial ^\infty A}$ maps endpoints of walls to endpoints of walls (see \cite[Lemma 3.11]{X}), so the isometry $F_A$ respects the tessalation of the apartment $A$, i.e. sends chambers in $A$ isometrically onto chambers in $A$. But a priori, we might have two different apartments $A, A'$ with the property that $F_A$ and $F_{A'}$ send a given chamber to two distinct chambers. So in order to build a global isometry from $(\tilde X, \tilde g_0)$ to $(\tilde X, \tilde g_1)$, we still need to check that the collection of maps $\{F_A\}_{A\in \mathcal A}$ are compatible.

Given any two apartments $A, A' \in \mathcal A$ with non-empty intersection $A \cap A' = K$, we want to check that the maps $F_A$ and $F_{A'}$ coincide on the set $K$. Let us first consider the case where $K$ is a half-space, i.e. there is a single wall $\gamma$ lying in $A\cap A'$, and $K$ coincides with the subset of $A$ (respectively $A'$) lying to one side of $\gamma$. In this special case, it is easy to verify that $F_A$ and $F_{A'}$ restrict to the same map on $K$. Indeed, Bourdon constructs the map $F_A$ as follows: given a point $p\in K$ take any two geodesics $\eta, \xi$ passing through $p\in (A, \tilde g_0)$, look at the corresponding pair of geodesics $\eta', \xi'$ in $(A, \tilde g_1)$ (obtained via the boundary map), and define $F_A(p):= \eta'\cap \xi'$. Bourdon argues that this intersection is non-empty, and independent of the choice of pairs of geodesics. The map $F_{A'}$ is defined similarly. But now if $p\in \text{Int}(K)$, one can choose a pair of geodesics $\eta, \xi \subset \text{Int}(K)$. Since $\text{Int}(K)$ is contained in both $A, A'$, this pair of geodesics can be used to see that $F_A(p)= \eta' \cap \xi' = F_{A'}(p)$. This shows that, if $K= A\cap A'$ is a half-space, then $F_A|_K \equiv F_{A'}|_K$.

For the general case, we now assume that we have a pair of apartments $A, A'$ with the property that $A\cap A'=K$ contains a chamber, and let $x$ be an interior point of this chamber. Then work of Hersonsky and Paulin \cite[Lemma 2.10]{hersonsky-paulin} produces a sequence of apartments $\{A_i\}_{i\in \mathbb N}$ with the property that $A_1=A$, each $A_i\cap A_{i+1}$ is a half-space containing $x$, and the $A_i$ converge to $A'$ in the topology of uniform convergence on compacts. From the discussion in the previous paragraph, one concludes that $F_A(x)= F_{A_i}(x)$ for all $i\in \mathbb N$, and from the uniform convergence, it is easy to deduce that $F_{A'(x)} = \lim F_{A_i}(x) = F_A(x)$. This verifies that the maps $\{F_A\}_{A\in \mathcal A}$ all coincide on a full-measure set (the interior points to chambers), and hence patch together to give a global isometry $F: (\tilde X, \tilde g_0)\rightarrow (\tilde X, \tilde g_1)$. Equivariance of the isometry follows easily from the naturality of the construction, along with the geometric nature of the maps $F_A$. Descending to the compact quotient completes the proof of Theorem \ref{thm:MLS-rigidity-special-case}.
\end{proof}

\begin{rem}
The argument presented here relies crucially on Bourdon's result in \cite{Bou1}. In the proof of the latter, the normalization of the spaces under consideration is important. The hyperbolic space mapping in must have curvature which matches the upper bound on the curvature in the target space. This is the key reason why the argument presented here does not immediately work in the setting of the {\bf Main Theorem}, where the metric $g_1$ is assumed to have piecewise curvature $\geq -1$.
\end{rem}

%
%%
%%%
%%%%
%%%%%%%%%%%%%%%%%%%%%%%%%%%%%%%%%%%%%%%%%%%%%%%%%%%%%%%%%%%%%

\section{$\mathcal M_{\equiv}^v(X)$ minimizes the volume}

This section is dedicated to proving Theorem \ref{thm:minimizing-the-volume}. For a vertex $v$ in our building, let $lk(v)$ denote the link of the vertex. Combinatorially, this link is a generalized $m$-gon, hence a $1$-dimensional spherical building.  The edges of the generalized $m$-gon correspond to the chamber angles at $v$, and so any piecewise Riemannian metric on the building induces a metric on the link:
\[ d_i: E(lk(v)) \to \mathbb{R}^+ .\nonumber\]
For these metrics, $Vol( -,d_i)$ is simply the sum of all edge lengths. We first argue that the vertex regularity hypothesis strongly constrains the angles.

\begin{lem}\label{combinatorial-angles}
Let $\mathcal G$ be a thick generalized $m$-gon. Assume we have a metric $d$ on $\mathcal G$ with the property that every $2m$-cycle in $\mathcal G$ has length exactly $2\pi$. Then every edge has length $\pi/m$.
\end{lem}

\begin{proof}

Consider a pair of vertices $v, w$ in $\mathcal G$ at combinatorial distance $=m$. Let $\mathcal P$ denote the set of all paths of combinatorial length $m$ joining $v$ to $w$. Note that, since any two paths in $\mathcal P$ have common endpoints at $v,w$, they cannot have any other vertices in common -- for otherwise one would find a closed loop of length $<2m$, which is impossible. The concatenation of any two paths in $\mathcal P$ form a $2m$-cycle, so has length exactly $2\pi$. By the thickness hypothesis, there are at least three such paths, hence every path in $\mathcal P$ has metric length $=\pi$. Applying this argument to all pairs of antipodal vertices in $\mathcal G$, we see that {\it every} path in $\mathcal G$ of combinatorial length $m$ has metric length $=\pi$.

Now let us return to our original pair $v,w$. Every edge emanating from $v$ can be extended to a (unique) combinatorial path of length $m$ terminating at $w$ (and likewise for edges emanating from $w$). This gives a bijection between edges incident to $v$ and edges incident to $w$. Let $e_i^w$ denote the edge incident to $w$ associated to the edge $e_i^v$ incident to $v$. Choosing $i\neq j$, we have a $2m$-cycle obtained by concatenating the paths $\mathbf{p}_i$ and $\mathbf{p}_j$ of combinatorial length $m$, joining $v$ to $w$ and passing through $e_i^v, e_j^w$. Within this $2m$-cycle, we have a path of combinatorial length
$m-1$ which can be extended, at each endpoint, by $e_i^v, e_j^w$ respectively. Since every path of combinatorial length $m$ has metric length exactly $\pi$, we see that the edges $e_i^v, e_j^w$ must have the same metric length. By the thickness hypothesis, we have $\deg(v)=\deg(w) \geq 3$, and it follows that every edge at the vertex $v$ has exactly the same metric length.

Using the same argument at every vertex, and noting that $\mathcal G$ is a connected graph, we see that every edge in $\mathcal G$ has exactly the same metric length. Finally, from the fact that the $2m$-cycles have length $=2\pi$, we see that this common length must be $=\pi/m$.
\end{proof}

Applying Lemma \ref{combinatorial-angles} to the links of each vertex in $X$, gives us:

\begin{cor}\label{metric-angles-equal-comb-angles}
If $g\in \mathcal{M}_{neg}^{v}$, then at every vertex $v\in X$, all the metric angles are equal to the combinatorial angles.
\end{cor}

Recall that the area of a hyperbolic (geodesic) polygon, by the Gauss-Bonnet formula, is completely determined by the number of sides and the angles at the vertices. So we also obtain:

\begin{cor}\label{constant-volume}
The volume functional is constant on the space $\mathcal{M}_{\equiv}^{v}(X)$.
\end{cor}
 
We are now ready to establish Theorem \ref{thm:minimizing-the-volume}

\begin{proof}[Proof of Theorem \ref{thm:minimizing-the-volume}] 

We will argue by contradiction. Assume we have a metric $g_1\in \mathcal M^v_\geq(X) \setminus \mathcal M^v_\equiv (X)$ with the property that $Vol(X, g_1)\leq Vol(X, g_0)$. Applying the Gauss-Bonnet theorem to any chamber $C$, we obtain for either metric that
\[ \int_C K_i dvol_i = -\pi(n-2) + \sum_{j=1}^n\theta^{(j)}_i\]
where $n$ is the number of sides for any chamber, $\theta_i^{(j)}$ are the interior angles of $C$, and $K_i$ is the curvature function for the metric $g_i$.  Denote
by $\mathcal P(X)$ the collection of chambers in $X$. For the whole space $X$, we have
\begin{equation}\label{gauss_bonnet}
	\sum_{C\in \mathcal P(X)} \int_C K_i dvol_i = -|\mathcal P(X)|\pi(n-2) + \sum_{C\in \mathcal P(X)} \sum_{j=1}^n\theta^{(j)}_i.
\end{equation}
Under the assumptions of the Theorem, we have
\begin{align*}
	\sum_{C\in \mathcal P(X)} \int_C K_0 dvol_0 & = \sum_{C\in \mathcal P(X)} \int_C -1 \hskip 2pt dvol_0 \\
			& = - Vol(X, g_0) \\
			& \leq - Vol(X, g_1) \\
			& =  \sum_{C\in \mathcal P(X)} \int_C -1 \hskip 2pt dvol_1 \\
			& < \sum_{C\in \mathcal P(X)} \int_C K_1 dvol_1.
\end{align*}
(The last inequality is strict, since from the assumption that $g_1\in  \mathcal M^v_\geq(X) \setminus \mathcal M^v_\equiv (X)$, there must be at least one interior point on some chamber where the curvature $K_1$ is greater than $-1$.) Since the quantity $-|\mathcal P(X)|\pi(n-2)$ is independent of the choice of metric, applying equation (\ref{gauss_bonnet}) gives us
\[ \sum_{C\in \mathcal P(X)} \sum_{j=1}^n\theta^{(j)}_0 <  \sum_{C\in \mathcal P(X)} \sum_{j=1}^n\theta^{(j)}_1. \]
But each of these two sums can be interpreted as $\sum_{v}Vol(lk(v), d_i)$ for the respective metrics. Hence, there must be at least one vertex $v$ whose $d_0$-volume is strictly smaller than its $d_1$-volume. But by Corollary \ref{metric-angles-equal-comb-angles}, the vertex regularity hypothesis forces the volumes of the links to be equal, a contradiction. This completes the proof of the Theorem \ref{thm:minimizing-the-volume}.
\end{proof}

%
%%
%%%
%%%%
%%%%%
%%%%%%
%%%%%%%
%%%%%%%%
%%%%%%%%%%%%%%%%%%%%%%%%%%%%%%%%%%%%%%%%%%%%%%%%%%%%%%%%%

\section{Geodesic flows and geodesic currents on Fuchsian buildings}\label{sec:flows and currents}

In this section, we set up the terminology needed for the proof of Theorem 3.

\subsection{Geodesic flow}
Let $\tilde X$ be a hyperbolic building, equipped with a $\textrm{CAT}(-\epsilon^2)$ metric $g$ for some $\epsilon>0$, and $X= \tilde X / \Gamma$ where $\Gamma \leq \aut(X)$ acts freely, isometrically, and cocompactly. We make the following definitions:

\begin{itemize}
	\item Let $G_g(\tilde X)$ be the set of unit-speed parametrizations of geodesics in $(\tilde X,g)$ equipped with the compact-open topology. Since $\tilde X$ is $\textrm{CAT}(-\epsilon^2)$, $G_g(\tilde X) \cong (\partial^\infty \tilde X\times \partial^\infty \tilde X \times \mathbb{R})\setminus (\Delta\times \mathbb R)$ where $\Delta$ is the diagonal in $\partial^\infty \tilde X\times \partial^\infty \tilde X$. The quotient space $G_g(X):=G_g(\tilde X)/\Gamma$ by the naturally induced $\Gamma$-action is the space of unit-speed geodesic parametrizations on $X= \tilde X/\Gamma$.  
	\item As in \cite[Section 3]{ball_brin}, let $S'$ denote the set of all unit length vectors based at a point in $X^{(1)}\setminus X^{(0)}$ (i.e. at an edge but not a vertex) and pointing into a chamber.  $S'C$ is the set pointing into a particular chamber $C$.  $S'_xC$ is the set pointing into $C$ and based at $x$.  $S'_x = \cup S'_xC_i$ is the union over all chambers adjacent to $x$.
	\item  For $v\in S'C$, let $I(v)\in S'C$ to be the vector tangent to the geodesic segment through C generated by $v$ and pointing the opposite direction.  Let $F(v) \subset S'$ be the set of all vectors based at the footpoint of $I(v)$ which geodesically extend the segment defined by $v$.  Let $W$ be the set of all bi-infinite sequences $(w_n)_{n\in \mathbb{Z}}$ such that $w_{n+1}\in F(w_n)$ for all $n$.
	\item Let $\sigma$ be the left shift on $W$.
	\item Let $t_v$ be the length of the geodesic segment in $C$ generated by $v$.
\end{itemize}

The geodesic flow on $G_g(\tilde X)$ is $g_t(\gamma(s)) = \gamma(s+t)$. It can also be realized by the suspension flow over $\sigma: W \to W$ with suspension function $\psi((w_n)) = t_{w_0}$.  Denote the suspension flow by $f_t: W_\psi \to W_\psi$ where 
\[W_\psi=\{((w_n),s): 0\leq s \leq \psi((w_n))\}/[((w_n),\psi((w_n)))\sim(\sigma((w_n)),0)]\]
and $f_t((w_n),s) = ((w_n),s+t)$. An explicit conjugacy between the suspension flow and the geodesic flow on the space $G_g'(X)$ of all geodesics which do not hit a vertex is as follows: $h: G_g'(X) \to W_\psi$ by $\psi(\gamma(t)) = ((w^\gamma_n), t^\gamma)$ where $(w^\gamma_n)$ is the trajectory of $\gamma$ through $S'$ indexed so that $w_0$ is $\dot\gamma(-t^\gamma)$ for $t^\gamma$ the smallest $t\geq 0$ for which $\dot\gamma(-t^\gamma)$ belongs to $S'$.

%\begin{rem}
%The spaces $G_g(\tilde X)$ and $G_g(X)$ are independent of the choice of metric on $X$. In contrast, the spaces $W$ and $G_g'(X)$ do a priori depend on the choice of metric.
%\end{rem}

%
%%
%%%
%%%%
%%%%%%%%%%%%%%%%%
%%%%%%%%%%%%%%%%%

\subsection{Liouville measure}\label{sec:liou}

We also want an analogue of Liouville measure.  We use the one constructed in \cite{ball_brin}.  On $S'$ define $\mu$ by 
\[ d\mu(v) = \cos\theta(v) d\lambda_x(v)dx \]
where $\theta(v)$ is the angle between $v$ and the normal to the edge it is based at, $\lambda_x$ is the Lebesgue measure 
on $S'_x$ and $dx$ is the volume on the edge.  This measure is invariant under $I$ by an argument well known from billiard 
dynamics (see e.g. \cite{CFS}). %The push-forward of $\mu$ via the inclusion map $\hat S' \to S'$ of $\mu$ induces a measure on $S'$.

Consider $W$ as the state space for a Markov chain with transition probabilities

\begin{equation*}
	p(v,w) = \left\{
	\begin{array}{cc}
		\frac{1}{|F(v)|} & \mbox{if } w\in F(v) \\
		0 & \mbox{else.}
	\end{array}\right.
\end{equation*}
Ballmann and Brin prove that $\mu$ is a stationary measure for this Markov chain (\cite{ball_brin} Prop 3.3) and hence $\mu$ 
induces a shift invariant measure $\mu^*$ on the shift space $W$. Under the suspension flow on $W_\psi$, $\mu^*\times dt$ is invariant. Using the conjugacy $h$, pull back this measure to $G_g'(X)\subset G_g(X)$ and denote the resulting geodesic flow-invariant measure induced on $G_g(X)$ by $L_g$. As Ballmann and Brin remark, $\mu\times dt$ is the Liouville measure on the interior of each chamber $C$, so $L_g$ is a natural choice as a Liouville measure analogue on $G_g(X)$.

We close this section with a quick remark about geodesics along walls, which will be used in the calculations of Section \ref{sec:compute}.

%  In $\mathcal{M}^v_*$, every wall is geodesic and every geodesic tangent to an edge remains in $X^{(1)}$, by Corollary \ref{metric-angles-equal-comb-angles}. Without the non-singular vertices condition (i.e. in $\mathcal{M}_{neg}$) it is possible that a geodesic might spend some time along a wall, leave it, then return to it or another wall later.  This sort of behavior introduces some technical issues in the argument of section \ref{vol_conject}, which we will note there. For now, we make the following observation: 

\begin{lem}\label{lem:walls zero}
Let $g$ be a metric in $\mathcal{M}_{neg}$. Let $T$ be the set of geodesics which are tangent to a wall at some point.  Then $L_g(T)=0$.
\end{lem}

\begin{proof}
By a standing assumption, each edge in $X$ is geodesic. Thus, any geodesic which is tangent to a wall at some point will hit a vertex. These geodesics are omitted in the construction of $L_g$, and hence form a zero measure set when we think of $L_g$ as a measure on all of $G_g(X)$.
\end{proof}

\vskip 10pt

%
%%
%%%
%%%%%%%%%%%%%%%%%%%%%%%%%%%%%%%%%%%%%

\subsection{Geodesic currents}\label{sec:currents}

\begin{defn}
Let $\mathscr{G}(\tilde X)$ denote the space of (un-parametrized and un-oriented) geodesics in $\tilde X$.  
\end{defn}

We note that for any negatively curved metric $g$, $\mathscr{G}(\tilde X) = G_g(\tilde X)/\sim$, where $\gamma \sim \eta$ if they agree up to a reparametrization. We equip $\mathscr{G}(\tilde X)$ with the quotient topology induced from $G_g(\tilde X)$. We have a $\Gamma$-equivariant identification 
\[\mathscr{G}(\tilde X) \cong [(\partial^\infty\tilde X \times \partial^\infty\tilde X) \setminus \Delta]/[(\xi_1,\xi_2)\sim(\xi_2,\xi_1)],\]
so $\mathscr{G}(\tilde X)$ is independent of the choice of metric.

\begin{rem}
At several points below, we will deal with elements of $\mathscr{G}(\tilde X)$ by representing them by elements of $G_g(\tilde X)$. We adopt the notational convention that if $c\in \mathscr{G}(\tilde X)$, then $\bar c$ denotes a geodesic in $G_g(\tilde X)$ representing it. The choice of the metric $g$ will either be explicit, or clear from context.
\end{rem}

\begin{defn}
A \emph{geodesic current} on $X = \tilde X/\Gamma$ is a positive Radon measure on $\mathscr{G}(\tilde X)$ which is $\Gamma$-invariant and cofinite (recall that a Radon measure is a Borel measure which is both inner regular and locally finite). Let $\mathscr{C}(X)$ denote the space of geodesic currents.  We equip $\mathscr{C}(X)$ with the weak-$^*$ topology, under which it is complete (see, e.g. Prop. 2 of \cite{bonahon_geometry}).
\end{defn}

\begin{exmp} The following are geodesic currents on a compact Fuchsian building quotient $X$ which will play a role in our later proofs:
\begin{itemize}
	\item Any geodesic flow-invariant Radon measure on $G_g(\tilde X)/\Gamma$ induces a geodesic current on $X$, so $L_g$ induces the Liouville current, also denoted $L_g$. The construction of the Liouville measure gives the following \emph{local expression} for the Liouville current. For any $g$-geodesic segment $\sigma$, parametrize the geodesics transversal to it by $(x,\theta,(w_n))$ where $x$ is the point of intersection with $\sigma$, $\theta$ is the angle between the geodesic direction and the normal to $\sigma$ at this point, and $(w_n)$ is the sequence in $W$ to which the geodesic corresponds. Then
	\[ dL_g = \cos\theta d\theta dx d\nu \]
	where $\nu$ is the Markov measure with the transition probabilities described in the previous subsection.
	\item For any primitive closed geodesic $\alpha$ in $X$, the sum of Dirac masses on each element of the $\Gamma$-orbit of $\tilde \alpha$ is a geodesic current, denoted by $\langle \alpha \rangle$.  
	\item For a non-primitive closed geodesic $\beta = \alpha ^n$, define $\langle\beta\rangle := n\langle\alpha\rangle$.  
\end{itemize}
\end{exmp}

\begin{prop}\label{density}
Let $\mathcal{C}\subset \mathscr{C}(X)$ be the set of currents which are supported on a single closed geodesic.  (I.e., it consists of all positive multiples of the currents $\langle\alpha\rangle$ described above.)  Then $\mathcal{C}$ is dense in $\mathcal{C}(X/\Gamma)$.  
\end{prop}

\begin{proof} 
In \cite[Theorem 7]{bonahon_negative}, Bonahon establishes the analogous property for geodesic currents on $\delta$-hyperbolic groups, with a proof given in \cite[Section 3]{bonahon_negative}. Bonahon's argument makes use of the Cayley graph $Cay(G)$ of $G$, but only relies on negative curvature properties of the Cayley graph -- the group structure plays no role in the proof. A careful reading of the arguments shows that it applies verbatim in our setting.
\end{proof}

%
%%
%%%
%%%%
%%%%%
%%%%%%
%%%%%%%
%%%%%%%%
%%%%%%%%%%%%%%%%%%%%%%%%%%%%%%%%%
%%%%%%%%%%%%%%%%%%%%%%%%%%%%%%%%%

\section{Transversality}\label{sec:transverse}

The key tool in the proof of Theorem \ref{thm:MLS-determines-volume}, as in Otal's original work on MLS rigidity and Croke and Dairbekov's work on MLS and volume, is the intersection pairing for geodesic currents. This is a finite, bilinear pairing on the space of currents, which recovers the intersection number for geodesics when the currents in question are Dirac measures on closed geodesics, and can also recover lengths of closed geodesics and the total volume of the space. For surfaces it is defined by
\[ i(\mu,\lambda) = (\mu\times \lambda)(DG(X)) \]
where $DG(X)$ is the set of all transversally intersecting pairs of unparametrized geodesics on $X$.

The main problem in extending this tool to the building case is the fact that $DG(X)$ is not topologically or combinatorially defined for buildings. For a surface, transverse intersection of geodesics is detected by linking of their endpoints in the circle $\partial^\infty\tilde S$. This is no longer the case for buildings, and one can imagine a pair of geodesics in $\mathscr{G}(\tilde X)$ whose representatives as $g_0$-geodesics intersect, but whose $g_1$-geodesic representatives do not intersect due to the branching of the building (see, e.g., Figure \ref{fig:transv prob} below).

Therefore, we must introduce an adjusted version of the intersection pairing which uses transverse intersections which can be detected purely topologically or combinatorially. We will then prove that it retains enough of the necessary properties of $i(-,-)$ for our purposes.

We begin with a definition of transversality for geodesics in an apartment of $(\tilde X,g):$

\begin{defn}\label{defn:apt transversal}
Let $A$ be an apartment in $\tilde X$ and $c,d$ two geodesics in $\mathscr{G}(\tilde X)$ which are contained in $A$. We say $\gamma$ and $\eta$ are \emph{transversal in $A$} if the endpoints $\gamma(\pm \infty)$ and $\eta(\pm \infty)$ are distinct and linked in $\partial^\infty A$. (See Figure \ref{fig:apt transv}.)
\end{defn}

\begin{figure}[ht]
\begin{center}
\begin{tikzpicture}

\draw (0,0) ellipse (2 and 1);
\draw [line width =1, red] (1.4,-.7)--(-1.4,.7);
\draw [line width =1, blue] (-1.95,.2)--(-.45,.2);
\draw [line width =1, blue] (-.45,.2)--(.35,-.2);
\draw [line width =1, blue] (.35,-.2)--(1.95,-.2);

\node [red] at (-1.8,1) {$\gamma(-\infty)$} ;
\node [red] at (1.8,-1) {$\gamma(+\infty)$} ;

\node [blue] at (-2.6,.2) {$\eta(-\infty)$} ;
\node [blue] at (2.6,-.2) {$\eta(+\infty)$} ;

\node at (1.8,1) {$(A,g)$} ;

\end{tikzpicture}
\end{center}
\caption{$\gamma$ and $\eta$ are transversal in $A$. Their representatives in $G_g(\tilde X)$ are shown. For this metric, the geodesics meet at a large-angle vertex, share a segment, then diverge at a large-angle vertex.}\label{fig:apt transv}
\end{figure}
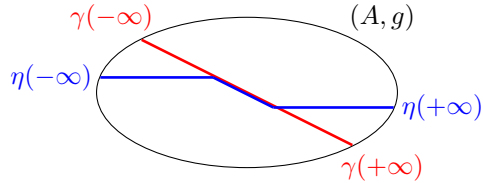

This definition is independent of $g$. We can (and sometimes will) apply this notion of transversality to pairs of geodesics in $G_g(\tilde X)$.

Note that for a particular metric $g$, if there are some vertices in $A$ surrounded by total angle $>2\pi$ it is possible that the $g$-realizations of two transversal geodesics meet at some vertex, agree along a segment, then diverge at a second vertex (as in Figure \ref{fig:apt transv}). Such behavior only happens along segments between vertices by our assumptions on the metric $g$.

With this in hand, we define two notions of transversality for geodesics in $\mathscr{G}(\tilde X)$. $N_\epsilon(K)$ denotes the $\epsilon$-neighborhood of the set $K$.

\begin{defn}\label{defn:transversal}
Let $\gamma, \eta \in \mathscr{G}(\tilde X)$. We say these geodesics are \emph{transversal for $g$} if for their $G_g(\tilde X)$ representatives $\bar \gamma$ and $\bar \eta$:
\begin{itemize}
	\item $\bar\gamma \cap \bar\eta \neq \emptyset$, and
	\item there exists some apartment $A$ in $\tilde X$ containing $\bar\gamma \cap \bar\eta$ such that for some $\epsilon>0$, $\bar\gamma \cap N_\epsilon(\bar\gamma \cap \bar\eta) \cap A$ and $\bar\eta \cap N_\epsilon(\bar\gamma \cap \bar\eta) \cap A$ are the intersections with $N_\epsilon(\bar\gamma \cap \bar\eta)$ of two transversal geodesics in $A$ in the sense of Definition \ref{defn:apt transversal}.
\end{itemize}
We write $\gamma \transv_g \eta$ if $\gamma$ and $\eta$ are transversal for $g$. (See Figure \ref{fig:g transv} for an illustration of $\gamma \transv_g \eta$.)
\end{defn}

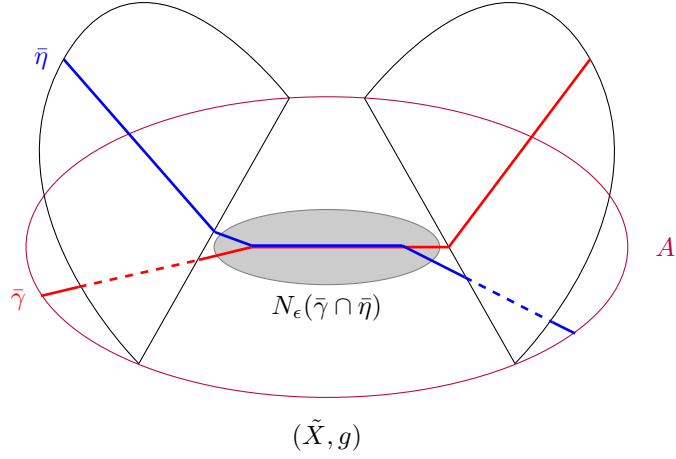
\begin{figure}[ht]
\begin{center}
\begin{tikzpicture}

\draw [gray=!40, fill=gray!40] (0,0) ellipse (1.5 and .5) ;

\draw [purple] (0,0) ellipse (4 and 2);
\draw  (2.5,-1.55)--(.5,1.98);
\draw  (-2.5,-1.55)--(-.5,1.98);

\draw plot [line width = 2, smooth, tension=2] coordinates { (.5,1.98) (3.5,2.5) (2.5,-1.55)};
\draw plot [line width = 1, smooth, tension=2] coordinates { (-.5,1.98) (-3.5,2.5) (-2.5,-1.55)};

\draw [line width = 1, red] (-3.8,-.65) -- (-3.3,-.53) ;
\draw [line width = 1, red, dashed] (-3.3,-.53) -- (-1.7,-.17) ;
\draw [line width = 1, red] (-1.7,-.17) -- (-1,0) ;
\draw [line width = 1, red] (-1,0) -- (1.62,0) ;
\draw [line width = 1, red] (1.62,0) -- (3.5,2.5) ;

\draw [line width = 1, blue] (-3.5,2.5) -- (-1.49,.2) ;
\draw [line width = 1, blue] (-1.49,.2) -- (-1,0.02) ;
\draw [line width = 1, blue] (-1,.02) -- (1,0.02) ;
\draw [line width = 1, blue]  (1,0.02) -- (1.83, -.4) ;
\draw [line width = 1, dashed, blue]  (1.83, -.4) -- (3,-1) ;
\draw [line width = 1, blue]  (3,-1) -- (3.3,-1.15);

\node at (0,-2.5) {$(\tilde X,g)$} ;
\node [purple] at (4.5,0) {$A$} ;
\node [red] at (-4.1,-.7) {$\bar \gamma$} ;
\node [blue] at (-3.8,2.5) {$\bar \eta$} ;
\node at (0,-.8) {$N_\epsilon(\bar \gamma \cap \bar \eta)$} ;

\end{tikzpicture}
\end{center}
\caption{$\gamma$ and $\eta$ are transversal for $g$, as they agree on $N_\epsilon(\bar\gamma \cap \bar \eta)$ with geodesics which are transversal in the apartment $A$. $\gamma$ and $\eta$ themselves are not transversal in any apartment.}\label{fig:g transv}
\end{figure}

We note that $\gamma \transv_g \eta$ is independent of the choice of the parametrizations of these geodesics, but is \emph{not} independent of the choice of $g$. In fact, it may be the case that two geodesics are transversal for $g_0$ but disjoint for $g_1$ (see Figure \ref{fig:transv prob}).

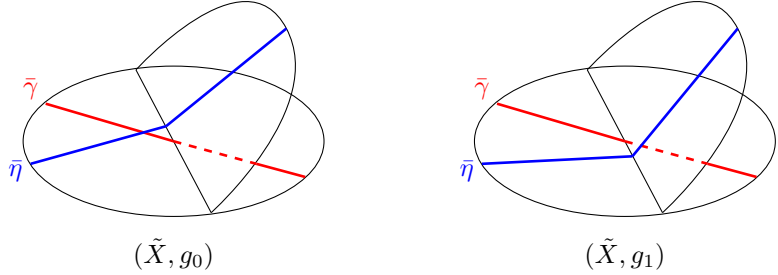
\begin{figure}[ht]
\begin{center}
\begin{tikzpicture}

\draw (-3,0) ellipse (2 and 1);
\draw (-2.5,-.95)--(-3.5,.95);
\draw plot [smooth, tension=2] coordinates { (-3.5,.95) (-1.5,1.5) (-2.5,-.95)};

\draw [line width=1, red] (-4.7,.5) -- (-3,0);
\draw [line width=1, red, dashed] (-3,0) -- (-1.9,-.3);
\draw [line width=1, red] (-1.9,-.3) -- (-1.25,-.475);

\draw [line width=1, blue] (-4.9,-.3) -- (-3.1,.2);
\draw [line width=1, blue] (-3.1,.2) -- (-1.5, 1.5);

\node [red] at (-4.9,.7) {$\bar \gamma$};
\node [blue] at (-5.1,-.4) {$\bar \eta$};
\node at (-3,-1.5) {$(\tilde X,g_0)$};

\draw (3,0) ellipse (2 and 1);
\draw (3.5,-.95)--(2.5,.95);
\draw plot [smooth, tension=2] coordinates { (2.5,.95) (4.5,1.5) (3.5,-.95)};

\draw [line width=1, red] (1.3,.5) -- (3,0);
\draw [line width=1, red, dashed] (3,0) -- (4.1,-.3);
\draw [line width=1, red] (4.1,-.3) -- (4.75,-.475);

\draw [line width=1, blue] (1.1,-.3) -- (3.1,-.2);
\draw [line width=1, blue] (3.1,-.2) -- (4.5, 1.5);

\node [red] at (1.1,.7) {$\bar \gamma$};
\node [blue] at (.9,-.4) {$\bar \eta$};
\node at (3,-1.5) {$(\tilde X,g_1)$};

\end{tikzpicture}
\end{center}
\caption{$\gamma$ and $\eta$ are transversal for $g_0$, but not for $g_1$.}\label{fig:transv prob}
\end{figure}

\begin{defn}\label{defn:Dg}
For a fixed metric $g$, let $D_g(\tilde X) \subset \mathscr{G}(\tilde X)\times \mathscr{G}(\tilde X)$ be the set of pairs $(\gamma, \eta)$ such that $\bar \gamma \transv_g \bar \eta$, where $\bar \gamma$ and $\bar \eta$ are any $G_g(\tilde X)$-representatives of $\gamma$ and $\eta$.
\end{defn}

Again, we emphasize that $D_g(\tilde X)$ depends on $g$.

For a notion of transversality which does not depend on $g$ we introduce the following:

\begin{defn}\label{defn:ess transv}
Let $(\gamma, \eta) \in \mathscr{G}(\tilde X) \times \mathscr{G}(\tilde X)$. We say that $\gamma$ and $\eta$ are \emph{essentially transversal} and write $\gamma \transv^* \eta$ if there exists some apartment $A\subset \tilde X$ containing $\gamma$ and $\eta$ such that $\gamma$ and $\eta$ are transversal in $A$ (as in Definition \ref{defn:apt transversal}).
\end{defn}

Being contained in an apartment and being transversal in an apartment do not depend on $g$, so $\transv^*$ is independent of $g$.

\begin{defn}\label{defn:Dg*}
Let $\mathscr{D}^*(\tilde X) \subset \mathscr{G}(\tilde X)\times \mathscr{G}(\tilde X)$ be the set of pairs $(\gamma, \eta)$ such that $\gamma \transv^* \eta$.
\end{defn}

We now collect a few simple but essential properties of the sets $D_g(\tilde X)$ and $\mathscr{D}^*(\tilde X)$.

\begin{lem}
For all $\alpha \in \Gamma$, $\alpha \cdot D_g(\tilde X) = D_g(\tilde X)$ and $\alpha \cdot \mathscr{D}^*(\tilde X) = \mathscr{D}^*(\tilde X)$.
\end{lem}

\begin{proof}
This follows from the definitions using (in the case of $D_g(\tilde X)$) that $\alpha$ is an isometry, and (in the case of $\mathscr{D}^*(\tilde X)$) that $\alpha$ is a combinatorial isomorphism.
\end{proof}

\begin{lem}
$D_g(\tilde X)$ and $\mathscr{D}^*(\tilde X)$ are symmetric in the sense that exchanging the two coordinates of any element produces another element in the set.
\end{lem}

\begin{proof}
This is clear from the definitions.
\end{proof}

\begin{lem}
Let $\phi: \tilde X_0 \to \tilde X_1$ be a combinatorial isomorphism Fuchsian buildings. Then $\phi(\mathscr{D}^*(\tilde X_0)) = \mathscr{D}^*(\tilde X_1)$.
\end{lem}

\begin{proof}
A combinatorial isomorphism maps apartments to apartments and preserves the linking of endpoints in an apartment. The result is then immediate from the definition of $\mathscr{D}^*(\tilde X)$.
\end{proof}

%
%%
%%%
%%%%
%%%%%
%%%%%%%%%%%%%%%%%%%%%%%%%%%%%%%%%%%%%%%%%%%%%%%%

\section{Intersection pairing(s)}\label{sec:pairings}

Corresponding to our two notions of transverse geodesics, we introduce two definitions of the intersection pairing for geodesic currents.

\begin{defn}\label{defn:std pairing}
Fix a metric $g$ on $X$ and let $\mu, \nu \in \mathscr{C}(X)$. The $\Gamma$-invariant measures $\mu$ and $\nu$ descend to finite measures $\bar \mu$ and $\bar \nu$ on $\mathscr{G}(\tilde X)/\Gamma$. The \emph{intersection pairing of $\mu$ and $\nu$ with respect to $g$} is
\[ i_g(\mu,\nu) := (\bar\mu \times \bar\nu)(D_g(\tilde X)/\Gamma). \]
Equivalently, if we fix a measurable fundamental domain $\mathscr{F}$ for the action of $\Gamma$ on $D_g(\tilde X)$,
\[ i_g(\mu, \nu) := (\mu \times \nu)(\mathscr{F}).\]
\end{defn}

Because of the role played by $g$ in defining $D_g(\tilde X)$ this pairing depends on $g$. For this reason it is not the pairing we want to use. To build a pairing which depends on $g$ only through some possible dependence of $\mu$ and $\nu$ on $g$ we must make some modifications.

\begin{defn}\label{defn:weighting}
Define $\varpi: \mathscr{D}^*(\tilde X) \to \mathbb{N}$ as follows. Fix an apartment $A$ containing $\gamma$ and $\eta$ and let $\mathcal{W}(\gamma, \eta)$ be the set of all walls $w$ in $A$ which are transversal in $A$ to both $\gamma$ and $\eta$. Note that $w\transv^* \gamma$ and $w \transv^* \eta$. Let
\[ \varpi(\gamma, \eta) := \prod_{w \in \mathcal{W}(\gamma, \eta)} (q(w)-1) \]
if $\mathcal{W}(\gamma, \eta)$ is nonempty and $\varpi(\gamma, \eta):=1$ if $\mathcal{W}(\gamma, \eta)$ is empty. Recall that $q(w)$ is the multiplicity of the wall -- the number of chambers containing any edge in $w$.
\end{defn}

To verify that $\varpi$ is well-defined, we need to prove the following two lemmas.

\begin{lem}\label{lem:transv}
If $A$ and $A'$ are apartments of $\tilde X$ in which $\gamma$ and $\eta$ intersect transversally, and $\mathcal{W}$, $\mathcal{W}'$ are the corresponding sets of walls transversal to the pair, then there is a bijective, multiplicity-preserving map between $\mathcal{W}$ and $\mathcal{W}'$.
\end{lem}

\begin{proof}
We have noted above that the $\transv^*$ condition which specifies which walls are in $\mathcal{W}$ and $\mathcal{W}'$ is independent of the choice of metric on $\tilde X$. Since $\tilde X$ is a Fuchsian building, we can fix a hyperbolic metric $g_0$ on $\tilde X$, that is, a metric in $\mathcal{M}_\equiv^v(\tilde X)$. Choosing this metric simplifies the geometry we use in the following argument.

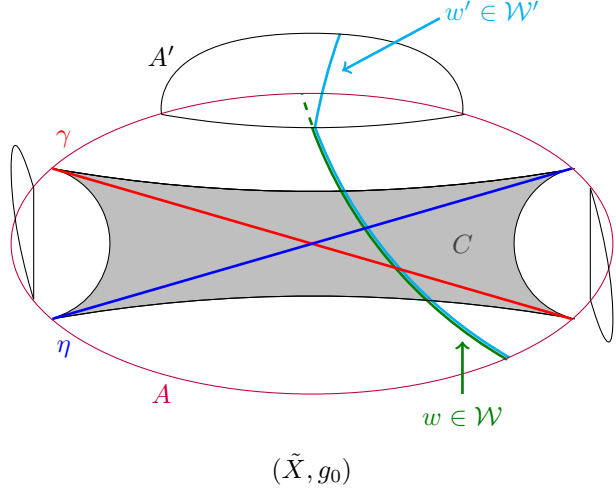
\begin{figure}[ht]
\begin{center}
\begin{tikzpicture}

\draw [fill=gray!50] (0,0) ellipse (4 and 2);

\draw [fill=white] (-3.45,1) arc (260:280:20) ;
\draw [white, fill=white] (-3.45,.99) -- (-3.45,2.2) -- (3.45,2.2) -- (3.45,.99) ;
\draw (-3.45,1) arc (260:280:20) ;
\draw [fill=white] (-3.45,-1) arc (100:80:20) ;
\draw [white, fill=white] (-3.45,-.99) -- (-3.45,-2.2) -- (3.45,-2.2) -- (3.45,-.99) ;
\draw [fill=white] (-3.45,1) arc (75:-75:1.03) ;
\draw [white, fill=white] (-3.44,1) -- (-3.44,-1) -- (-4.1,-1) -- (-4.1,1) ;
\draw [fill=white] (3.45,1) arc (105:255:1.03) ;
\draw [white, fill=white] (3.44,1) -- (3.44,-1) -- (4.1,-1) -- (4.1,1) ;

\draw [green!50!black, line width=1] (0,1.55) arc (200:240:5.9);
\draw [dashed, green!50!black, line width=1] (0,1.58) arc (200:195.8:5.9);
\draw [cyan, line width=1] (0.04,1.55) arc (200.3:240:5.9);
\draw [cyan, line width=1] (0.04,1.55) arc (170:160:7.4);

\draw  (-3.45,-1) arc (100:80:20) ;
\draw  (-3.45,1) arc (260:280:20) ;

\draw [purple] (0,0) ellipse (4 and 2);

\draw [line width=1, red] (-3.45,1) -- (3.45,-1);
\draw [line width=1, blue] (-3.45,-1) -- (3.45,1);

\draw (-3.7,-.74) -- (-3.7,.74);
\draw plot [smooth, tension=2] coordinates { (-3.7,-.74) (-4,1) (-3.7,.74)};

\draw (3.7,-.74) -- (3.7,.74);
\draw plot [smooth, tension=2] coordinates { (3.7,-.74) (4,-1) (3.7,.74)};
\draw (-2,1.72) arc(260:280:11.6) ;
\draw plot [smooth, tension=2] coordinates { (-2,1.72) (0,2.8) (2,1.72)};

\node at (0,-3) {$(\tilde X, g_0)$} ;
\node [purple] at (-2,-2) {$A$} ;
\node at (-2,2.5) {$A'$} ;

\draw [line width=1, green!50!black, ->] (2,-2) -- (2,-1.3) ;
\node [green!50!black] at (2,-2.3) {$w \in \mathcal{W}$} ;

\draw [line width=1, cyan, ->] (1.7,3) -- (.4,2.3) ;
\node [cyan] at (2.4,3.1) {$w' \in \mathcal{W'}$} ;

\node [red] at (-3.3,1.4) {$\gamma$} ;
\node [blue] at (-3.3,-1.4) {$\eta$} ;

\node [gray!50!black] at (2,0) {$C$} ;

\end{tikzpicture}
\end{center}
\caption{An essentially transversal pair $(\gamma, \eta)$ belonging to both $A$ and $A'$. The walls $w$ and $w'$ correspond under the combinatorial isomorphism between $A$ and $A'$ from the proof of Lemma \ref{lem:transv}.}\label{fig:wall correspondence}
\end{figure}

Let $C$ be the convex hull in $A$ of $\gamma \cup \eta$ (see Figure \ref{fig:wall correspondence}). Since apartments are convex sets, $C\subset A'$. Since $\gamma \transv^* \eta$ and $A$ is isometric to $\mathbb{H}^2$, $C$ has non-empty interior. (For other metrics, with large angles around vertices, this may not hold, hence our choice to work with $g_0$.) Therefore $C$ contains a point from the interior of some chamber $c$. Since $A$ and $A'$ are full unions of chambers, $c \subset A \cap A'$. Then by the third building axiom, there is a combinatorial isomorphism from $A$ to $A'$ fixing $A \cap A'$. This combinatorial isomorphism preserves walls, their multiplicities, transversality within apartments, $\gamma(\pm \infty)$ and $\eta(\pm \infty)$, and hence $\gamma, \eta$. Therefore it induces the desired map between $\mathcal{W}$ and $\mathcal{W}'$.
\end{proof}

\begin{lem}\label{lem:wall finite}
For any $(\gamma, \eta) \in \mathscr{D}^*(\tilde X)$, $\mathcal{W}(\gamma, \eta)$ is finite.
\end{lem}

\begin{proof}
Let $A$ be an apartment in which $\gamma$ and $\eta$ are transversal, and recall that this transversality is independent of the choice of the metric on this apartment. Fix $g_0\in \mathcal{M}_\equiv^v(\tilde X)$ for which all chambers are isometric. Consider $\bar\gamma$ and $\bar\eta$ in this metric.

With respect to $g_0$, $\bar \gamma$ and $\bar \eta$ cross at a non-zero angle. Let $\bar c_i$ be the geodesics joining an endpoint of $\gamma$ to an endpoint of $\eta$. Using some hyperbolic geometry, there is a finite $R$ such that for all $i$, $d_{g_0}(\bar \gamma \cap \bar \eta, \bar c_i)<R$. Therefore, any $w\in\mathcal{W}(\gamma, \eta)$ is a wall passing through $B_R(\bar \gamma \cap \bar \eta)$. There are only finitely many of these, since all chambers are isometric. (See Figure \ref{fig:wall finite}.)
\end{proof}

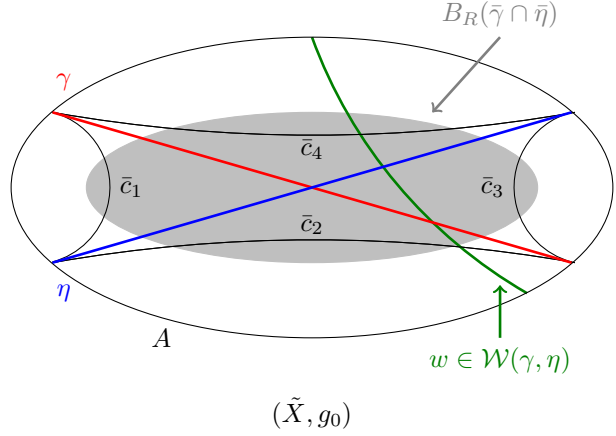
\begin{figure}[ht]
\begin{center}
\begin{tikzpicture}

\draw [gray!50, fill=gray!50] (0,0) ellipse (3 and 1) ;

\draw (-3.45,1) arc (260:280:20) ;
\draw (-3.45,-1) arc (100:80:20) ;
\draw (-3.45,1) arc (75:-75:1.03) ;
\draw (3.45,1) arc (105:255:1.03) ;

\draw [green!50!black, line width=1] (0,2) arc (200:240:6.5);

\draw  (-3.45,-1) arc (100:80:20) ;
\draw  (-3.45,1) arc (260:280:20) ;

\draw  (0,0) ellipse (4 and 2);

\draw [line width=1, red] (-3.45,1) -- (3.45,-1);
\draw [line width=1, blue] (-3.45,-1) -- (3.45,1);

\node at (0,-3) {$(\tilde X, g_0)$} ;
\node  at (-2,-2) {$A$} ;

\draw [line width=1, green!50!black, ->] (2.5,-2) -- (2.5,-1.3) ;
\node [green!50!black] at (2.5,-2.3) {$w \in \mathcal{W}(\gamma,\eta)$} ;

\draw [line width=1, gray, ->] (2.5,2) -- (1.6,1) ;
\node [gray] at (2.5,2.3) {$B_R(\bar \gamma \cap \bar \eta)$} ;

\node [red] at (-3.3,1.4) {$\gamma$} ;
\node [blue] at (-3.3,-1.4) {$\eta$} ;
\node at (-2.4,0) {$\bar c_1$};
\node at (2.4,0) {$\bar c_3$};
\node at (0,-.5) {$\bar c_2$};
\node at (0,.5) {$\bar c_4$};

\end{tikzpicture}
\end{center}
\caption{Construction for the proof of Lemma \ref{lem:wall finite}.}\label{fig:wall finite}
\end{figure}

We now prove some lemmas on the structure of $\mathscr{D}^*(\tilde X)$ and $\varpi$.

\begin{lem}\label{lem:closed}
$\mathscr{D}^*(\tilde X)$ is a closed subset of $(\mathscr{G}(\tilde X) \times \mathscr{G}(\tilde X)) \setminus Diag^*$ where $Diag^*$ is the following `generalized diagonal':
\[ Diag^* := \{ (\gamma, \eta) \in \mathscr{G}(\tilde X) \times \mathscr{G}(\tilde X) :  \gamma(\pm\infty) \cap \eta(\pm\infty) \neq \emptyset\}.\]
\end{lem}

Recall that the topology on $\mathscr{G}(\tilde X)$ is the quotient topology induced by the compact open topology on $G_g(\tilde X)$ for any metric $g$.

\begin{proof}

Suppose that $(\gamma_n, \eta_n) \in \mathscr{G}(\tilde X) \times \mathscr{G}(\tilde X)$ and that $(\gamma_n, \eta_n) \to (\gamma^*, \eta^*)$ with $(\gamma^*, \eta^*)\notin Diag^*$. Let $A_n$ be an apartment in $\tilde X$ in which $\gamma_n$ and $\eta_n$ are transversal. Fix some basepoint $*$ in $\tilde X$. For a fixed $R>0$, there are only finitely many possibilities for $A_n \cap B_R(*)$. Therefore, by a subsequence argument, we can construct a subsequence $(n_i)$ such that $A_{n_i} \cap B_R(*)$ is constant for all $i>R$.  Let $A^* = \bigcup_{R>0} (A_{n_{R+1}} \cap B_R(*))$. Since $A^*$ agrees with an apartment on any ball around $*$, $A^*$ is an apartment. In addition $A^*$ contains the sequence $(\gamma_n \cap A^*, \eta_n\cap A^*)$ which converges to $(\gamma^*, \eta^*)$ on any compact subset of $A^*$. Since $A^*$ is closed, $\gamma^*$ and $\eta^*$ lie in $A^*$. 

The endpoints of $\gamma_n$ and $\eta_n$ approach the endpoints of $\gamma$ and $\eta$. Since the endpoints of $\gamma_n$ and $\eta_n$ are linked, the endpoints of $\gamma^*$ and $\eta^*$ must be linked unless they degenerate so that some endpoint of $\gamma^*$ agrees with some endpoint of $\eta^*$. As $(\gamma^*, \eta^*)\notin Diag^*$, this does not happen. This proves the lemma.

\end{proof}

\begin{cor}
$\mathscr{D}^*(\tilde X)$ is a measurable set.
\end{cor}

\begin{prop}\label{prop:msrbl}
$\varpi$ is a lower semicontinuous function on $\mathscr{D}^*(\tilde X)$. In particular, it is measurable.
\end{prop}

\begin{proof}

We need to show that if $(\gamma_n, \eta_n) \to (\gamma^*, \eta^*)$, then $\liminf_{n\to\infty} \varpi(\gamma_n, \eta_n) \geq \varpi(\gamma^*, \eta^*)$.

Suppose $(\gamma_n, \eta_n) \to (\gamma^*, \eta^*)$. Since $\mathscr{D}^*(\tilde X)$ and $\varpi$ are independent of the choice of metric, we are free to fix a hyperbolic metric $g_0$ on $\tilde X$ and represent elements $c$ of $\mathscr{G}(\tilde X)$ by geodesics $\bar c$ in $G_{g_0}(\tilde X)$.

Suppose that $\bar\gamma^*, \bar\eta^* \subset A^*$ and that $\bar w^*$ is a wall in $A^*$ transversal to $\bar\gamma^*$ and $\bar\eta^*$. Since we are working with a hyperbolic metric, $\bar\gamma^*$ and $\bar\eta^*$ intersect $\bar w^*$ at nonzero angles. For geodesics or geodesic segments in $A^*$, the property of crossing at a non-zero angle is an open condition. Therefore, for all sufficiently large $n$, $\bar\gamma_n\cap A^*$ and $\bar\eta_n \cap A^*$ cross $\bar w^*$ at nonzero angles. If $A_n$ is an apartment containing $\bar\gamma_n$ and $\bar\eta_n$, then there is a wall $\bar w_n'$ in $A_n$ which agrees with $\bar w^*$ on the intersection of $A_n$ with $A^*$, which contains, in particular, the intersection of this wall segment with $\bar \gamma_n$ and $\bar \eta_n$. Then $\bar\gamma_n$ and $\bar\eta_n$ are transversal to $\bar w'_n$ in $A_n$. The fact that these three geodesics are in the common apartment $A_n$ gives us that $\gamma_n \transv^* w'_n$ and $\eta_n \transv^* w'_n$.

Since $\bar w^*$ and $\bar w_n'$ agree on the intersection of $A^* \cap A_n$ (which has nontrivial interior, as in the proof of Lemma \ref{lem:transv}), $q(w^*) = q(w_n')$. Applying this to all $w^* \in \mathcal{W}(\gamma^*,\eta^*)$, we get $\varpi(\gamma_n, \eta_n) \geq \varpi(\gamma^*, \eta^*)$ for sufficiently large $n$. The result follows.

\end{proof}

We cannot upgrade this result to continuity for $\varpi$. The precise manner in which continuity fails is investigated in more detail in Lemma \ref{lem:bdry}. Figure \ref{fig:lower semi-cty}, which illustrates that proof, also provides an illustration of how $\varpi(\gamma^*, \eta^*)$ may be strictly less than $\varpi(\gamma_n, \eta_n)$.

$\mathscr{D}^*(\tilde X)$ and $\varpi$ are also $\Gamma$-invariant:

\begin{lem}\label{lem:Gamma inv}
For any $\alpha \in \Gamma$, $\alpha \cdot \mathscr{D}^*(\tilde X) = \mathscr{D}^*(\tilde X)$ and $\varpi(\gamma, \eta) = \varpi(\alpha \cdot \gamma, \alpha \cdot \eta)$.
\end{lem}

\begin{proof}
This is clear, as $\gamma$ is a simplicial automorphism.
\end{proof}

We are now prepared to define a modified version of the intersection pairing which will reproduce some of the important properties of $i_g(-,-)$, but which will be independent of the metric $g$.

\begin{defn}\label{defn:mod int}
Let $\mu, \nu \in \mathscr{C}(X)$. We define the \emph{combinatorial intersection pairing} of $\mu$ and $\nu$ by
\[ \hat \imath(\mu, \nu) := \int _{\mathscr{D}^*(\tilde X)/\Gamma} \varpi(\gamma, \eta) d\mu d\nu \]
where, by Lemma \ref{lem:Gamma inv}, $\varpi$ descends to a function on $\mathscr{D}^*(\tilde X)/\Gamma$. Equivalently, if $\mathscr{F}$ is a fundamental domain for the $\Gamma$ action on $\mathscr{D}^*(\tilde X)$, 
\[ \hat \imath(\mu, \nu) := \int_\mathscr{F} \varpi(\gamma, \eta)d\mu d\nu. \]
\end{defn}

%
%%
%%%
%%%%
%%%%%
%%%%%%%%%%%%%%%%%%%%%%%%%%%%%%%%%%%%%%%

\section{Computing intersection pairings}\label{sec:compute}

We compute the intersection pairings with the most geometric interest, namely those between closed geodesic currents $\langle \alpha \rangle$ and the Lebesgue currents $L_g$. We begin with the pairings by $i_g(-,-)$.

First we note that if $\bar c$ and $\bar d$ are $g$-geodesics in $X$, then the connected components of $\bar c \cap \bar d$ are either points, or nontrivial closed segments of the geodesics. In the latter case, the geodesic segment joins two points on the 0- or 1-skeleton of $X$.

\begin{prop}\label{prop:int count}
Let $\alpha, \beta \in \pi_1(X)=\Gamma$ be prime elements. Let $\bar \alpha$ and $\bar \beta$ be the $g$-geodesics in the free homotopy classes of $\alpha$ and $\beta$. Consider the connected components $p_i$ of $\bar \alpha \cap \bar \beta$. For each $i$, let $\tilde p_i$ be a lift of $p_i$ to $\tilde X$ and $\tilde \alpha_i, \tilde \beta_i$ be the lifts of $\bar \alpha$ and $\bar \beta$ through $\tilde p_i$. Then $i_g(\langle \alpha \rangle, \langle \beta \rangle)$ is the number of $p_i$ such that $\tilde \alpha_i \transv_g \tilde \beta_i$.
\end{prop}

\begin{proof}

Recall, $i_g(\langle \alpha \rangle, \langle \beta \rangle) = (\langle \alpha \rangle \times \langle \beta \rangle)(D_g(\tilde X)/\Gamma)$. Since $\langle \alpha \rangle$ and $\langle \beta \rangle$ are supported solely on lifts of $\bar \alpha$ and $\bar \beta$ (or rather, the elements of $\mathscr{G}(\tilde X)$ which these $g$-geodesics represent), only pairs of lifts of $\bar \alpha$ and $\bar \beta$ have non-zero measure. Since we are measuring pairs mod $\Gamma$ we have one such pair for each intersection $p_i$ of $\bar\alpha$ and $\bar\beta$ in $X$. Since we are measuring only $D_g(\tilde X)/\Gamma$, the only pairs with non-zero measure are those which lift to a pair in $D_g(\tilde X)$, i.e., the pairs $\tilde \alpha_i \transv_g \tilde\beta_i$. Each such pair gives $(\langle \alpha \rangle \times \langle \beta \rangle)$-measure one, proving the result.

\end{proof}

For a metric $g$ and curve $c$, write $l_g(c)$ for the $g$-length of $c$.

\begin{prop}\label{prop:int length}
Let $\alpha \in \pi_1(X) = \Gamma$, and let $\bar \alpha$ be the $g$-geodesic in $X$ in this free homotopy class. Write $\bar\alpha$ as a union of segments $s_i$ such that each segment either has its interior in the interior of a chamber, or is a wall segment joining two vertices. Then 
\[i_g(\langle \alpha \rangle, L_g) = \sum_i q(s_i)l_g(s_i)\]
where $q(s_i)$ is 2 if $s_i$ is in the interior of a chamber, and is the multiplicity of the wall if $s_i$ lies along a wall.
\end{prop}

\begin{proof}
It is sufficient to prove the result for prime closed geodesics.

The support of $\langle \alpha \rangle \times L_g$ in $D_g(\tilde X)$ consists of pairs $(\tilde \alpha,c)$ in $\mathscr{G}(\tilde X)$, represented by $\tilde \alpha, \bar c$ in $G_g(\tilde X)$, where $\tilde \alpha$ is a lift of $\bar \alpha$ and $c \transv_g \tilde \alpha$. From its local description it is clear that $L_g$ assigns zero measure to the set of $c$ for which $\bar c$ is tangent to $\tilde \alpha$, as there is no angular spread to such geodesics. Therefore, we can restrict our attention to those $c$ for which $\bar c$ intersects $\tilde \alpha$ at a positive angle. Further, we can ignore those $c$ for which $\bar c$ intersects $\tilde \alpha$ at a vertex (by Lemma \ref{lem:walls zero}) or at a point where $\tilde \alpha$ crosses a wall $w$ at a positive angle, as the basepoints of such geodesics form a discrete set. Therefore we consider only those pairs where $\alpha$ and $\bar c$ meet at a positive angle in the interior of a chamber, and those pairs where $\bar c$ meets a segment of $\tilde \alpha$ which lies along a wall at a positive angle.  Finally, since we are measuring $D_g(\tilde X)/\Gamma$, we need only consider a single lift $\tilde \alpha$ and those $\bar c$ which intersect it along a fundamental domain $F$ for the action of $\alpha \in \Gamma$ on $\tilde \alpha$, i.e., a segment of length $l_g(\alpha)$.

For a segment $s$ of $F$ in the interior of a chamber,
\begin{align} 
	(\langle \alpha \rangle \times L_g)&(\{(\tilde \alpha, c): \bar c \mbox{ non-singular and meets $s$ at a positive angle}\}) \nonumber \\
		&= L_g(\{(\tilde \alpha, c): \bar c \mbox{ non-singular and meets $s$ at a positive angle}\}) \nonumber 
\end{align}
At any point along $s$, there are only countably many angles measured from $s$ which correspond to singular geodesics since there are countably many vertices in $\tilde X$. Then from the local description of $L_g$ we can compute this measure as
\[ \int_{(p,\theta)\in s \times (-\frac{\pi}{2}, \frac{\pi}{2})} \cos \theta d\theta dp = 2 l_g(s)  \]
since any $c$ for which $\bar c$ meets $s$ at a positive angle satisfies $c \transv_g s$.

Now let $s$ be a segment in $F$ along a wall $w$. Again, every $c$ such that $\bar c$ hits $s$ at a positive angle satisfies $c \transv_g s$ since there is an apartment containing the wall segment $s$ and the two chambers on either side of it that $\bar c$ traverses. Then for each of the $q(w)$ chambers $C_i$ adjoining $s$, all $\bar c$ starting in $C_i$, passing through $s$ and continuing into some $C_j$ with $j\neq i$ are $\transv_g$ to $s$. By the calculation of the first part of the proof, together with the definition of $L_g$, the measure of these pairs for each (unordered) pair $\{i,j\}$ with $i\neq j$ is $2l_g(s)\frac{1}{q(w)-1}$.  There are $\frac{q(w)(q(w)-1)}{2}$ such pairs, giving a contribution of $q(w)l_g(s)$ to $i_g(\langle \alpha \rangle, L_g)$ for the segment $s$. This completes the proof.
\end{proof}

The argument of Proposition \ref{prop:int length} shows a fact we will need below:

\begin{cor}\label{cor:int length}
Let $s$ be any $g$-geodesic segment which does not lie along a wall. Then
\[ L_g(\{c: \bar c \mbox{ meets } s \mbox{ at a positive angle}\}) = 2l_g(s).\]
In particular, if $\bar \alpha$ has no segments along walls, then $i_g(\langle \alpha \rangle, L_g) = 2 l_g(\bar \alpha)$.
\end{cor}

Finally, we compute

\begin{prop}\label{prop:int vol}
For any $g$, $i_g(L_g, L_g) = 4\pi Vol_g(X)$.
\end{prop}

\begin{proof}

First, we note that the set of all pairs of geodesics $(c,d)$ in $\mathscr{G}(\tilde X) \times \mathscr{G}(\tilde X)$ such that $\bar c\cap \bar d$ is a positive length segment has $(L_g\times L_g)$-measure zero by Fubini's theorem, since for any fixed $\bar c$, the set of $\bar d$ tangent to it at some point is easily seen to have $L_g$-measure zero from the local description of $L_g$. Therefore we only need measure those pairs $(c,d)\in D_g(\tilde X)$ where $\bar c$ and $\bar d$ intersect at nonzero angle. By Lemma \ref{lem:walls zero}, the set of geodesics tangent to any wall has $L_g$ measure zero, so we omit these from our considerations as well. Finally, the set of pairs intersecting at a point on wall has $(L_g\times L_g)$-measure zero, as can be seen by fixing $c$ and then using the local description of $L_g$. Therefore, to compute $i_g(L_g,L_g)$ we need only measure those pairs $(c,d)$ whose $g$-representatives $\bar c$ and $\bar d$ intersect at positive angle in the interior of some chamber.

Second, since we are measuring $D_g(\tilde X)/\Gamma$, we can pick one lift of each chamber to $\tilde X$ and measure the set of all $(c,d)$ with $\bar c \transv_g \bar d$ at a point in the interior of such a chamber. Noting that $Vol_g(X) = \sum_C Vol_g(C)$ where the sum runs over all chambers in $X$, it is sufficient to prove 
\[(L_g\times L_g)(\{(c,d): \bar c \transv_g \bar d \mbox{ at a point in } Int(C) \}) = 4\pi Vol_g(C).\]

Let $S^+C$ be the set of inward-pointing unit tangent vectors based at non-vertex points in the boundary of $C$. By Santal\'o's formula (see \cite[\S19.5]{santalo}), 
\[ Vol_g(C) = \frac{1}{2\pi}\int_{v\in S^+C} l_g(\bar c_v) \cos\theta(v) d\theta dp \]
where $\bar c_v$ is the $g$-geodesic segment in $C$ generated by $v$, $\theta(v)$ is the angle between $v$ and the normal vector to the wall it lies on, and $p$ is the basepoint of $v$.  In addition, by Corollary \ref{cor:int length},
\[ l_g(c_v) = \frac{1}{2} \int_{d \in A_v} \cos \phi_v(d) d\phi_v dq \]
where $A_v$ is the set of $g$-geodesic segments $d$ in $C$ intersecting $\bar c_v$ at a positive angle, $\phi_v(d)$ is the angle between the normal to $\bar c_v$ and $d$ and $q$ is the intersection point. From these computations, and using the local description of $L_g$, we see immediately that
\begin{align} 
	 4\pi Vol_g(C) &=  \int_{v\in S^+C} \int_{d \in A_v} \cos \phi_v(d) \cos \theta(v) d\phi_v dq d\theta dp \nonumber \\
			& = (L_g \times L_g)(\{(c,d): \bar c \transv_g \bar d \mbox{ in } Int(C)\}). \nonumber
\end{align}

\end{proof}

We now turn to computing the combinatorial intersection pairing of these same currents.

\begin{prop}\label{prop:hat int count}
Let $\alpha, \beta \in \pi_1(X)=\Gamma$. Let $\hat \alpha$ and $\hat \beta$ be representative curves in the corresponding free homotopy classes which minimize the cardinality of $\hat \alpha \cap \hat \beta$. For each intersection $p_i$, pick a lift $\tilde p_i \in \tilde X$ and lift the curves to $\tilde \alpha_i, \tilde \beta_i$ through $\tilde p_i$. Using the endpoints at infinity of these curves, we can consider $\tilde \alpha_i$ and $\tilde \beta_i$ as elements of $\mathscr{G}(\tilde X)$. Then 
\[\hat \imath(\langle \alpha \rangle, \langle \beta \rangle) = \sum_i \varpi(\tilde \alpha_i, \tilde \beta_i).\]
\end{prop}

\begin{rem}
The ``metric-free'' statement of the Proposition is possible because $\hat \imath(-,-)$ depends solely on the combinatorics of the building,
so is independent of metric.
\end{rem}

\begin{proof}
This result follows from the argument used to prove Proposition \ref{prop:int count}, with $\transv^*$ replacing $\transv_g$, and then incorporating the factor $\varpi(\tilde \alpha_i, \tilde \beta_i)$. Note that an intersection $p_i$ of $\hat \alpha$ and $\hat \beta$ with $\tilde \alpha_i$ and $\tilde \beta_i$ in a common apartment can be removed by a free homotopy if and only if the endpoints of the lifted geodesics at infinity are not linked.
\end{proof}

Our computations involving $L_g$ are aided by the following Lemma. We say two geodesics \emph{agree locally around $p$} if they agree in some neighborhood of $p$.

\begin{lem}\label{lem:mult lemma}
Fix a metric $g$ on $X$ and a geodesic $c \in \mathscr{G}(\tilde X)/\Gamma$ with $G_g(\tilde X)/\Gamma$-representative $\bar c$. Let $\hat A_n = \{ d \in \mathscr{G}(\tilde X)/\Gamma: \varpi(c,d)=n\}$ and $A_n =\{ d \in \mathscr{G}(\tilde X)/\Gamma: \bar d \transv_g \bar c \mbox{ and } \bar d \mbox{ locally agrees with } \bar \gamma \mbox{ for } \gamma \in \hat A_n \mbox{ around some } p\in \bar c \cap \bar\gamma \}$. Then
\[ L_g(A_n) = n L_g(\hat A_n). \]
\end{lem}

\begin{proof}

The local description of $L_g$ shows that $g$-geodesics representing elements of $\hat A_n$ or $A_n$ which share a segment with $\bar c$ have $L_g$-measure zero. We omit them from our calculations, and consider only geodesics which cross $\bar c$ at a non-zero angle. We can also omit any singular geodesics, since they have $L_g$-measure zero.

Write $\hat A_n = \bigsqcup_{\mathcal{W}} \hat A_{n,\mathcal{W}}= \bigsqcup_{\mathcal{W}} \{ d: \mathcal{W}(c,d)=\mathcal{W}\}$ as the disjoint union over all wall sets $\mathcal{W}(c,d)$ which appear for elements of $\hat A_n$. By our finiteness result Lemma \ref{lem:wall finite} and the fact that we are working in $\mathscr{G}(\tilde X)/\Gamma$, this is a finite union. Let $A_{n, \mathcal{W}}$ be the set of all $d$ whose $g$-geodesics representative $\bar d$ agrees locally with some element of $\hat A_{n,\mathcal{W}}$ around its intersection with $\bar c$.

Using $g$-geodesic representatives of our geodesics, it is clear that $A_{n,\mathcal{W}}$ differs from $\hat A_{n, \mathcal{W}}$ precisely by containing geodesics $\bar d'$ which agree with an element $\bar d$ of $\hat A_{n,\mathcal{W}}$ over some initial segment containing its intersection with $\bar c$ and then (perhaps) diverge from $\bar d$ by branching at a wall in $\mathcal{W}$ (see Figure \ref{fig:mult lemma}). (This uses the fact that we are considering only nonsingular geodesics.) At each wall $w\in \mathcal{W}$, the $L_g$ measure of $\hat A_{n,\mathcal{W}}$ relative to that of $A_{n,\mathcal{W}}$ inherits a factor $1/(q(w)-1)$ due to the Markov chain portion of the construction of $L_g$. The product of these factors is $1/\varpi(c,d)=1/n$ for any $d\in \hat A_{n,\mathcal{W}}$. Therefore, 
\[ L_g(A_{n,\mathcal{W}}) = n L_g(\hat A_{n,\mathcal{W}}).\]
Summing this over all $\mathcal{W}$ gives the desired result.

\end{proof}

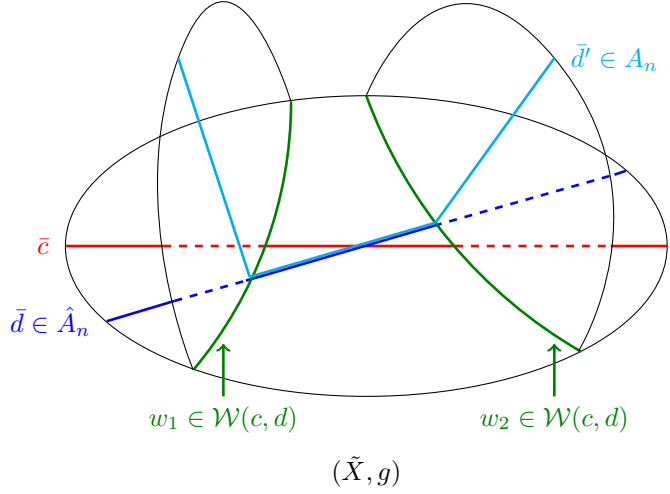
\begin{figure}[ht]
\begin{center}
\begin{tikzpicture}

\draw (0,0) ellipse (4 and 2);

\draw [line width=1, red] (-4,0) -- (-2.7,0);
\draw [line width=1, red] (3.25,0) -- (4,0);
\draw [line width=1, red] (-1.35,0) -- (1.18,0);
\draw [line width=1, red, dashed] (-2.7,0) -- (-1.35,0);
\draw [line width=1, red, dashed] (1.18,0) -- (3.25,0);

\draw [line width=1, blue] (-3.45,-1) -- (-2.55,-.73);
\draw [line width=1, blue, dashed] (-2.55,-.73) -- (-1.6,-.46);
\draw [line width=1, blue] (-1.6,-.46) -- (.9,.27);
\draw [line width=1, blue, dashed] (.9,.27) -- (3.3,.95);
\draw [line width=1, blue] (3.3,.95) -- (3.45,1);

\draw [green!50!black, line width=1] (0,2) arc (200:240:6.5);
\draw [green!50!black, line width=1] (-1,1.92) arc (0:-40:5.55);

\draw plot [smooth, tension=2] coordinates { (0,2) (2.5,2.5) (2.82,-1.4)};
\draw plot [smooth, tension=2] coordinates { (-1,1.92) (-2.5,2.5) (-2.3,-1.65)};

\draw [cyan, line width=1] (-2.5,2.5) -- (-1.55,-.41) ;
\draw [cyan, line width=1] (-1.55,-.41) -- (.92,.31);
\draw [cyan, line width=1] (.92,.31) -- (2.5,2.5);

\node at (0,-3) {$(\tilde X, g)$} ;

\draw [line width=1, green!50!black, ->] (2.5,-2) -- (2.5,-1.3) ;
\node [green!50!black] at (2.5,-2.3) {$w_2 \in \mathcal{W}(c,d)$} ;
\draw [line width=1, green!50!black, ->] (-1.9,-2) -- (-1.9,-1.3) ;
\node [green!50!black] at (-1.9,-2.3) {$w_1 \in \mathcal{W}(c,d)$} ;
\node [cyan] at (3.3,2.5) {$\bar d' \in A_n$} ;

\node [red] at (-4.3,0) {$\bar c$} ;
\node [blue] at (-4.2,-1) {$\bar d \in \hat A_n$} ;

\end{tikzpicture}
\end{center}
\caption{Illustrating Lemma \ref{lem:mult lemma}: A geodesic $\bar d$ in $\hat A_n$ and one geodesic $\bar d' \in A_n$ which agrees locally around $\bar d$ around its intersection with $\bar c$ but which is not in $\hat A_n$. $\bar d$ and $\bar d'$ can only differ by diverging at walls in $\mathcal{W}(c,d)$ such as $w_1$ and $w_2$.}\label{fig:mult lemma}
\end{figure}

With this we can complete our other two computations.

\begin{prop}\label{prop:hat int length}
Fix a metric $g$ and let $\alpha \in \pi_1(X)=\Gamma$ with $g$-geodesic representative $\bar \alpha$. Write $\bar \alpha$ as a union of segments $s_i$ which either have their interior in the interior of a chamber or are a wall segment joining two vertices. Then
\[ \hat \imath(\langle \alpha \rangle, L_g) = \sum_i q(s_i) l_g(s_i) \]
where $q(s_i)=2$ if $s_i$ is in the interior of a chamber and is the multiplicity of the wall if $s_i$ lies along a wall.

That is, 
\[ \hat \imath (\langle \alpha \rangle, L_g) = i_g(\langle \alpha \rangle, L_g).\]
\end{prop}

\begin{proof}
As before, we can restrict our attention to nonsingular geodesics throughout this proof.

Let $\hat A_n(\alpha) = \{d \in \mathscr{G}(\tilde X)/\Gamma: \varpi(\alpha,d)=n\}$. Let $A_n(\alpha) =\{ d \in \mathscr{G}(\tilde X)/\Gamma: \bar d \transv_g \bar \alpha \mbox{ and } \bar d \mbox{ locally agrees with } \bar \gamma \mbox{ for } \gamma \in \hat A_n(\alpha) \mbox{ around some } p\in \bar \alpha \cap \bar\gamma \}$. It is easy to verify that $\{d\in \mathscr{G}(\tilde X)/\Gamma:\bar d \transv_g \bar\alpha\} = \bigcup_{n>0} A_n(\alpha)$. We can prove this union is disjoint as follows. If $\bar d$ locally agrees around $p\in \bar \alpha \cap \bar d$ with $\bar \gamma_1$ for $\gamma_1 \in \hat A_{n_1}(\alpha)$ and with $\bar \gamma_2$ for $\gamma_2\in \hat A_{n_2}(\alpha)$, then $\bar \gamma_1$ and $\bar\gamma_2$ locally agree around $p$. Since $\gamma_2 \transv^* \alpha$, $\bar \gamma_2$ cannot diverge from $\bar\gamma_1$ until after it has crossed every wall in $\mathcal{W}(\alpha, \gamma_1)$, else there would be no common apartment containing $\alpha$ and $\gamma_2$. (We use here that these are nonsingular geodesics, so divergence only happens by branching at a wall, not at a large-angle vertex.) Similarly, $\bar \gamma_1$ must cross every wall in $\mathcal{W}(\alpha, \gamma_2)$. Thus $\mathcal{W}(\alpha, \gamma_1) = \mathcal{W}(\alpha, \gamma_2)$ and so $n_1=\varpi(\alpha, \gamma_1) = \varpi(\alpha, \gamma_2)=n_2$.

Using $g$-geodesic representatives for geodesics in $\mathscr{G}(\tilde X)/\Gamma$ when necessary, we calculate using Lemma \ref{lem:mult lemma} and the fact that the $A_n(\alpha)$ are disjoint:

\begin{align}
	\hat \imath(\langle \alpha \rangle, L_g) & = \int_{\mathscr{D}^*(\tilde X)/\Gamma} \varpi(\eta, \gamma) d\langle \alpha \rangle dL_g \nonumber \\
						& = \sum_{n>0} n  L_g(\hat A_n(\alpha)) \nonumber \\
						& = \sum_{n>0} L_g(A_n(\alpha)) \nonumber \\
						& = L_g(\{d \in \mathscr{G}(\tilde X)/\Gamma : \bar d \transv_g \bar\alpha \}) \nonumber \\
						& = i_g(\langle \alpha \rangle, L_g), \nonumber
\end{align}
using the arguments of Proposition \ref{prop:int length} at the last step. In this computation we have again ignored those $\bar d$ which lie along some segment of $\bar \alpha$ since they have $L_g$-measure zero. 

The result then follows from the expression for $i_g(\langle \alpha \rangle, L_g)$ given in Proposition \ref{prop:int length}.

\end{proof}

\begin{cor}\label{cor:hat int length}
Fix $g$. If $\alpha\in \pi_1(X)=\Gamma$ and if a proportion $\rho$ of $\bar \alpha$ lies along walls, then 
\[2l_g(\bar\alpha) \leq \hat \imath(\langle \alpha \rangle, L_g) \leq (2 + \rho q )l_g(\bar \alpha)\]
where $q$ is the maximum multiplicity of a wall in $\tilde X$. In particular, if $\bar \alpha$ has no segments along walls, then
\[\hat \imath(\langle \alpha \rangle, L_g) = 2 l_g(\bar \alpha).\]
\end{cor}

Finally,

\begin{prop}\label{prop:hat int vol}
For any $g$, $\hat \imath(L_g, L_g) =  4\pi Vol_g(X)$.
\end{prop}

\begin{proof}
The proof follows the proof of Proposition \ref{prop:int vol} essentially verbatim now that we have Corollary \ref{cor:hat int length} to replace Corollary \ref{cor:int length}.
\end{proof}

%
%%
%%%
%%%%
%%%%%
%%%%%%%%%%%%%%%%%%%%%%%%%%%%%%%%%%%%%%%%%%%%%%%%

\section{A geometric lemma}\label{sec:geom lemma}

We will need the following geometric lemma to prove the continuity result we want for the combinatorial intersection pairing $\hat \imath(-,-)$. For a geodesic $c\in \mathcal G(\tilde X)$, let $\bar c$ denote the $g$-geodesic representative of $c$. If $c$ is periodic, then we denote by $\hat c$ the periodic $g$-geodesic in $X=\tilde X/\Gamma$ obtained by projecting $\bar \gamma$.

\begin{lem}\label{lem:exp small}
Fix a metric $g$ on $X$ and a periodic geodesic $\gamma \in \mathcal G(\tilde X)$.  Define the sets 
\[\hat W_n = \{ \eta \in \mathscr{G}(\tilde X)/\Gamma : \varpi(\gamma, \eta) >n\}\]
\begin{align}
	W_n = \{\eta \in \mathscr{G}(\tilde X)/\Gamma: \ &\bar \eta \transv_g \bar \gamma  \mbox{ and } \bar \eta \mbox{ locally agrees with } \nonumber \\
				&\bar c \mbox{ for } c \in \hat W_n \mbox{ around some } p\in \bar \gamma \cap \bar c\}. \nonumber
\end{align}
Then there is some constant $\beta>1$, depending only on the metric $g$, such that 
\[ L_g(W_n) \leq \beta^{-n} l_g(\hat \gamma).\]
\end{lem}

\begin{proof}

We have noted previously that the local expression for $L_g$ implies that the set of all geodesics which are tangent to $\bar \gamma$ at some point have $L_g$-measure zero, so we consider only those $\bar \eta$ which intersect $\bar \gamma$ at a nonzero angle. Similarly, those $\bar \eta$ which intersect $\bar \gamma$ at a vertex have $L_g$-measure zero, so we consider only intersections at non-singular points. Therefore the crossing angle between the geodesics is well-defined and for such $\eta$, $\bar \eta \transv_g \bar \gamma$.

Since $(\tilde X, g)$ has a compact quotient, there are only finitely many isometry classes of chambers in $\tilde X$. Therefore, any $g$-geodesic segment $\bar c$ in $\tilde X$ of length $L$ crosses at most $DL$ walls, for some constant $D>0$ which depends only on the metric $g$.

Let $A$ be an apartment containing $\bar \gamma$ and $\bar \eta$. Suppose that $\bar \gamma$ and $\bar \eta$ meet at angle $\theta\leq \frac{\pi}{2}$. As noted in the proof of Lemma \ref{lem:wall finite}, if we let $\bar c_i$ be the $g$-geodesics connecting an endpoint of $\bar \gamma$ to and endpoint of $\bar \eta$, then any wall in $\mathcal{W}(\gamma, \eta)$ must lie to the side of $\bar c_i$ which contains the intersection of $\bar \gamma$ and $\bar \eta$. That is, each such wall must intersect at least one of the geodesic segments $\bar s_i$ which connect the intersection point $\bar \gamma \cap \bar \eta$ to the nearest point on $\bar c_i$.

$A$ equipped with the metric $g$ is a $\textrm{CAT}(-\epsilon^2)$ space for some $\epsilon>0$ depending only on $g$, since $g$ descends to a negatively curved metric on the compact quotient $X$. Using some comparison geometry and some standard calculations in hyperbolic geometry, one can bound the angle by
\[ \theta \leq C e^{- l_g(\bar s_i)\alpha} \  \mbox{ for all } i=1,2,3,4\]
where $l_g(\bar s_i)$ is the $g$-length of $\bar s_i$. $C$ and $\alpha$ are positive constants depending only on $-\epsilon^2$, and therefore only on $g$.

Now suppose that $\varpi(\gamma, \eta)>n$. If $q^*+1$ is the maximum multiplicity of any wall in $\tilde X$, when there must be at least $n/q^*$ walls in $\mathcal{W}(\gamma, \eta)$. Since each wall crosses at least one of the segments $\bar s_i$, there are at most $4DL$ such walls, where $L$ is the maximum length of the four segments $\bar s_i$. Therefore
\[ \frac{n}{q^*} \leq |\mathcal{W}(\gamma, \eta)| \leq 4DL \ \mbox{ and so } \ L \geq \frac{n}{4Dq^*}. \]
Combining this with our bound on $\theta$ in terms of the lengths $l_g(\bar s_i)$, we get
\[ \theta \leq C e^{- L\alpha } \leq C e^{-\frac{n\alpha }{4Dq^*}}.\]
This angle bound holds not just for the pair $(\bar \gamma, \bar \eta)$, but also clearly holds for $(\bar \gamma, \bar c)$ where $\bar c$ locally agrees with $\bar \eta$. That is, it holds for $c \in W_n$.

Now $L_g(W_n)$ can be computed in local coordinates using a small geodesic segment along $\bar \gamma$ to define the local coordinates. The bound on $\theta$ tells us that the total angular spread of all $\eta \in W_n$ intersecting $\bar \gamma$ at a particular point $p$ is exponentially small in $n$. The local expression for the measure 
\[ dL_g = \cos \theta d\theta dp \]
then integrates to something exponentially small in $n$ on performing the $\theta$ integration, and gives a term proportional to $l_g(\hat \gamma)$ when integrating over those $p\in \bar \gamma$ which lie in a fundamental domain for the action of $\Gamma$ on $\tilde X$. This proves the result. 
\end{proof}

%
%%
%%%
%%%%
%%%%%
%%%%%%%%%%%%%%%%%%%%%%%%%%%%%%%%%%%%%%%%%%%%%%%%

\section{Continuity at $L_g$}\label{sec:cty}

One of the key properties of the intersection pairing for surfaces is that it is continuous with respect to the weak-* topology on $\mathscr{C}(X)$. We now want to investigate one special case of this continuity which persists for the pairing $\hat \imath(-,-)$.

We want to prove the following:

\begin{prop}\label{prop:cty}
Let $g$ and $g'$ be metrics in $\mathcal{M}_{neg}(\tilde X)$. Let $(\mu_k)$ be a sequence of currents in $\mathscr{C}(X)$ which are of the form $c_k\langle \alpha_k \rangle$ for $\alpha_k \in \pi_1(X)$. Then
\[ \mu_k \overset{weak-*}\longrightarrow L_g \ \implies \ \hat \imath(\mu_k, L_{g'}) \longrightarrow \hat \imath(L_g, L_{g'}).\]
\end{prop}

This asserts a very specific continuity of the pairing at $L_g$.  To prove this, we first need a result on the sets $\varpi^{-1}(n)$.

 \begin{lem}\label{lem:bdry}
 For all $n>0$, and all metrics $g, g'$, 
 \[ (L_g\times L_{g'})(\partial \varpi^{-1}(n)) =0. \]
 \end{lem}

\begin{proof}

Let $\hat B_n = \varpi^{-1}(n,\infty)$; since $\varpi$ is lower semicontinuous (Proposition \ref{prop:msrbl}), these are open subsets of $\mathscr{D}^*(\tilde X)$. Then $\varpi^{-1}(n) = \hat B_{n-1} \setminus \hat B_n$. Therefore, $\partial \varpi^{-1}(n) \subset \partial \hat B_{n-1} \cup \partial \hat B_n$ so it is sufficient to prove that 
\[ (L_g \times L_{g'})(\partial \hat B_n)=0 \ \mbox{ for all }n.\]

As $\hat B_n$ is open, if $(\gamma, \eta) \in \partial \hat B_n$, then 
\begin{itemize}
	\item $\varpi(\gamma, \eta) \leq n$, and
	\item there is a sequence $(\gamma_k, \eta_k)\to (\gamma, \eta)$ in $\mathscr{D}^*(\tilde X)$ with $\varpi(\gamma_k, \eta_k)>n$ for all $k$.
\end{itemize}
Since $\varpi(\gamma_k,\eta_k)>\varpi(\gamma, \eta)$, for each $k$, there exists some  wall $w_k$ satisfying $w_k \transv^* \gamma_k$ and $w_k\transv^*\eta_k$, but either $w_k \ntransv^* \gamma$ or $w_k \ntransv^*\eta$. By passing to a subsequence, we can without loss of generality assume that $w_k\ntransv^*\eta$ for all $k$.

Fix a hyperbolic metric on $\tilde X$, and represent each geodesic $c$ in $\mathscr{G}(\tilde X)$ with its $G_{g_0}(\tilde X)$ representative $\bar c$. Let $p_k = \bar w_k \cap \bar\eta_k$ and $q_k = \bar w_k \cap \bar \gamma_k$.  First, let us consider the case where $p_k$ and $q_k$ remain in some compact subset of $\tilde X$. Then we may pass to a sequence such that $p_k \to p^*\in \bar\eta$ and $q_k\to q^* \in \bar\gamma$. After again passing to a subsequence and using arguments as in Lemma \ref{lem:closed}, $\bar w_k$ must converge to a geodesic $\bar w^*$ which is $\transv_{g_0}$-transversal to $\bar \eta$ at $p^*$ and to $\bar \gamma$ at $q^*$.  Since this geodesic must locally agree with the limit of a sequence of walls near $p^*$, $\bar w^*$ must in fact be a wall.

Within any compact set, the set of wall segments is discrete, so the fact that $w_k \to w^*$ implies that as $k\to \infty$, $\bar w_k$ and $\bar w^*$ agree on larger and larger compact sets containing $p^*$ and $q^*$. Eventually such compact sets become so large that some $w_k$ agrees with $w^*$ past all of $w^*$'s intersections with the walls in $\mathcal{W}(\eta, w^*)$ and $\mathcal{W}(\gamma, w^*)$. Fix some such large $k^*$. Then it is possible to construct an apartment $A^*$ which agrees with $A_{k^*}$ on the convex hull of $\bar \eta, \bar \gamma$ and $\bar w^*$, where $A_{k^*}$ witnesses the $\transv^*$-transversality of these geodesics, and also contains $w_{k^*}$.

This shows that in fact $w_{k^*} \transv^* \eta$ and $w_{k^*} \transv^* \gamma$, a contradiction to our assumptions. We conclude that either $p_k$ or $q_k$ or both must tend to infinity, in the sense that they escape all compact sets. Again after passing to a subsequence we can assume $p_k\to \eta(+\infty)$ and $q_k\to q^*$, or $p_k \to p^*$ and $q_k \to \gamma(+\infty)$, or $p_k\to \eta(+\infty)$ and $q_k \to \gamma(+\infty)$. In any case, we now have the following description of $\partial \hat B_n$:

\

If $(\gamma, \eta)\in \partial\hat B_n$, there is a wall connecting an endpoint of one geodesic to a point on the other geodesic, or to one of its endpoints at infinity. (See Figure \ref{fig:lower semi-cty}.)

\

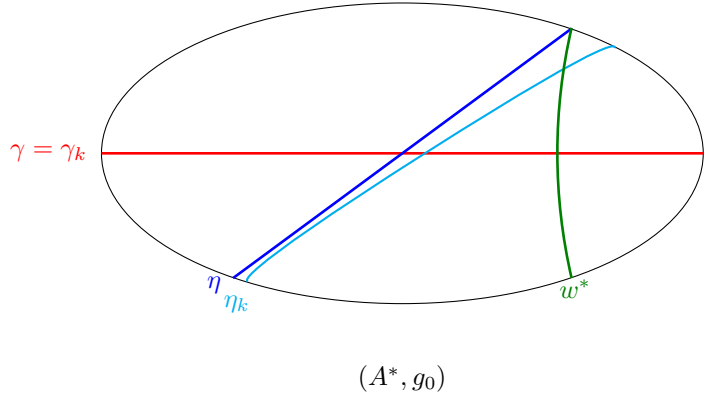
\begin{figure}[ht]
\begin{center}
\begin{tikzpicture}

\draw (0,0) ellipse (4 and 2);

\draw [line width=1, red] (-4,0) -- (4,0);
\draw [line width=1, blue] (-2.25,-1.66) -- (2.25,1.66);

\draw[thick, cyan] plot [smooth, tension=2] coordinates { (-2.05,-1.7) (.3,0) (2.82,1.4)};
\draw [green!50!black, line width=1] (2.25,-1.66) arc (193:167:7.4);

\node[red] at (-4.7,0) {$\gamma = \gamma_k$} ;
\node at (0,-3) {$(A^*, g_0)$} ;
\node[blue] at (-2.5,-1.75) {$\eta$} ;
\node[cyan] at (-2.2,-2) {$\eta_k$} ;
\node[green!50!black] at (2.3,-1.8) {$w^*$} ;

\end{tikzpicture}
\end{center}
\caption{A simple case illustrating a pair $(\gamma, \eta)$ in $\partial \hat B_n$ as described by Lemma \ref{lem:bdry}. This also demonstrates why $\varpi(\gamma, \eta)$ may be strictly smaller than $\lim \varpi(\gamma_k, \eta_k)$.}\label{fig:lower semi-cty}
\end{figure}

We prove that $(L_g\times L_{g'})(\partial \hat B_n)=0$ by fixing one geodesic, say $\eta$, and proving that the $L_g$-measure of the set of geodesics $\gamma$ such that $\gamma \transv^* \eta$ and $\gamma(\infty)=w(\infty)$ for some wall crossing or asymptotic to $\eta$ is zero.

From the definition of $L_g$ it is sufficient to show that in any apartment $A$ containing $\eta$, the set of such $\gamma$ has zero measure with respect to the measure given locally in coordinates by $\cos \theta d\theta dp$. But note that $A$ contains only countably many walls, so at any basepoint $p$, there are only countable many angles which will give a $g'$-geodesic forward asymptotic to such a wall. Thus the $d\theta$-measure of this set of angles is zero, giving the desired result.

\end{proof}

We are now ready to prove Proposition \ref{prop:cty}.

\begin{proof}[Proof of Proposition \ref{prop:cty}]

To simplify notation, we write $\mu_k$ for $c_k\langle \alpha_k\rangle$, recalling that $\mu_k \to L_g$ in the weak-* topology.  We write $\hat \alpha_k$ for the closed $g$-geodesic in $X$ to which this current is associated. 

Let $\varpi_n(\gamma, \eta) = \max\{\varpi(\gamma, \eta), n\}$. We then define
\[ a_{nk} = \int_{\mathscr{D}^*(\tilde X)/\Gamma} \varpi_n d\mu_k dL_{g'},\]
\[ a_{*k} = \int_{\mathscr{D}^*(\tilde X)/\Gamma} \varpi d\mu_k dL_{g'} = \hat \imath(\mu_k, L_{g'}),\]
\[ a_{n*} = \int_{\mathscr{D}^*(\tilde X)/\Gamma} \varpi_n dL_g dL_{g'},\]
\[ a_{**} = \int_{\mathscr{D}^*(\tilde X)/\Gamma} \varpi dL_g dL_{g'}= \hat \imath(L_g, L_{g'}).\]
Note that by our calculations in Section \ref{sec:compute}, all these integrals have finite values. We want to show that $a_{*k} \to a_{**}$ as $k\to \infty$.

Two limit statements involving the $a_{nk}$ are straightforward. First, the functions $\varpi_n$ converge pointwise to $\varpi$ with $0\leq \varpi_n \leq \varpi_{n+1}$, so by the monotone convergence theorem,
\[ a_{nk} \to a_{*k} \ \mbox{ and } \ a_{n*} \to a_{**} \ \mbox{ as } \ n\to \infty, \mbox{ for all }k.\]
Second, each $\varpi_n$ is the sum of finitely many characteristic functions for sets whose boundaries, by Lemma \ref{lem:bdry}, have $L_g$-measure zero. The weak-* convergence of $\mu_k$ to $L_g$ guarantees that the measures of such sets under $\mu_k$ converges to their measure under $L_g$ (see, e.g., \cite[\S1.2]{billingsley}). Then it is easy to see that
\[ a_{nk} \to a_{n*} \ \mbox{ as } \ k \to \infty, \mbox{ for all }n.\]

To prove $a_{*k}\to a_{**}$, we need to give a uniform rate of convergence for $a_{nk}$ in either $n$ or $k$. We do this for $n$ using Lemma \ref{lem:exp small}.

Recall that $\mu_k = c_k \langle \alpha_k \rangle$. We note that as $c_k \langle \alpha_k \rangle$ weak-* converges to $L_g$, $c_k \langle \alpha_k \rangle (\mathscr{G}(\tilde X)/\Gamma) \to L_g(\mathscr{G}(\tilde X)/\Gamma)$. Since this limit is fixed but $\langle \alpha_k \rangle(\mathscr{G}(\tilde X)/\Gamma)$ is proportional to the length of the closed geodesic $\hat \alpha_k$ in $X$ (in any metric, up to adjusting the constant of proportionality) we can conclude that there is some constant $b>0$ such that $c_k \leq \frac{b}{l_{g'}(\hat \alpha_k)}$.

Since $\varpi_n$ and $\varpi$ differ only on $\varpi^{-1}(n,\infty)$,
\[  \left| \int_{\mathscr{D}^*(\tilde X)/\Gamma} \varpi_n d\mu_k dL_{g'} - \int_{\mathscr{D}^*(\tilde X)/\Gamma} \varpi d\mu_k dL_{g'} \right| \leq c_k L_{g'}(W_n)  \]
where $W_n$ is (as in Lemma \ref{lem:exp small}) 
\begin{align} 
	W_n=\{\eta \in \mathscr{G}(\tilde X)/\Gamma \ : \ &\bar \eta \transv_{g'} \bar \alpha_k \mbox{ and } \bar \eta \mbox{ locally agrees with } \bar c \nonumber \\
		& \mbox{ around some } p\in \bar \alpha_k \cap \bar c \mbox{ with } \varpi(\alpha_k, c)>n\}. \nonumber
\end{align}
This relies again on Lemma \ref{lem:mult lemma} to relate the $L_{g'}$-measures of $\hat W_n$ and $W_n$.  Using Lemma \ref{lem:exp small} and our bound on $c_k$ we get that for all $n$ (and, crucially, uniformly in $k$),
\[ \left| \int_{\mathscr{D}^*(\tilde X)/\Gamma} \varpi_n d\mu_k dL_{g'} - \int_{\mathscr{D}^*(\tilde X)/\Gamma} \varpi d\mu_k dL_{g'} \right| \leq \frac{b}{l_{g'}(\hat \alpha_k)} \beta^{-n} l_{g'}(\hat \alpha_k) =  b \beta^{-n} \]
for constants $b>0$ and $\beta>1$ which depend only on $g'$. That is,
\[ | a_{nk} - a_{*k} | < b \beta^{-n} \ \mbox{ for all } k.\]

Finishing the proof is now straightforward. Let $\epsilon>0$ be given. Choose $N$ so that $n>N$ implies $b\beta^{-n}<\epsilon$. Then for all $k$ and all $n>N$, $|a_{nk} - a_{*k}|<\epsilon$. Since $a_{n*}\to a_{**}$, we can choose some $\hat n> N$ such that $|a_{\hat n *} - a_{**}|<\epsilon$. Given this $\hat n$, using the fact that $a_{\hat n k} \to a_{\hat n *}$, pick $K$ so large that $k>K$ implies $| a_{\hat n k} - a_{\hat n *}| < \epsilon.$ Since $\hat n >N$, $|a_{\hat n k} - a_{*k}|<\epsilon$ for all $k$. Combining these inequalities we have that for all $k>K$
\[ |a_{*k}-a_{**}| \leq |a_{*k}-a_{\hat n k}| + |a_{\hat n k}-a_{\hat n *}| + |a_{\hat n *}-a_{**}| < 3\epsilon\]
completing the proof.

\end{proof}

%
%%
%%%
%%%%
%%%%%
%%%%%%%%%%%%%%%%%%%%%%%%%%%%%%%%%%%%%%%%%%%%%%%%

\section{Marked length spectrum and volume}\label{sec:MLS-determines-volume}

We can now prove Theorem \ref{thm:MLS-determines-volume}.

\begin{thm}
Let $g_0$ and $g_1$ be metrics in $\mathcal{M}_{neg}(X)$. Suppose we have the following marked length spectrum inequality: for all $\alpha \in \pi_1(X) = \Gamma$, 
\[ l_{g_0}(\alpha) \leq l_{g_1}(\alpha). \]
Then $Vol_{g_0}(X) \leq Vol_{g_1}(X)$.
\end{thm}

\begin{proof}

By the length inequality assumption and Corollary \ref{cor:hat int length}, we have for any $\alpha_k \in \pi_1(X)=\Gamma$, and $c_k>0$,
\begin{equation}\label{eqn:length comp}
\begin{split}
\hat \imath(c_k\langle \alpha_k \rangle, L_{g_0}) &\leq (2+\rho_k q)c_kl_{g_0}(\alpha_k) \\ 
& \leq (2+\rho_k q)c_kl_{g_1}(\alpha_k) \\
& \leq \hat \imath(c_k\langle \alpha_k \rangle, L_{g_1})+\rho_kqc_kl_{g_1}(\alpha_k) 
\end{split}
\end{equation}
where $\rho_k$ is the proportion of time the closed $g_0$-geodesic $\bar\alpha_k$ lies along a wall and $q$ is the maximum multiplicity of any wall in $\tilde X$.

By the density of multiples of closed-geodesic currents in $\mathscr{C}(X)$, we can find a sequence $c_k\langle \alpha_k \rangle \to L_g$. As noted in the proof of Proposition \ref{prop:cty}, $c_k\leq \frac{b}{l_{g_0}(\hat\alpha_k)}$. Since $X$ is compact, the metrics $g_0$ and $g_1$ are Lipschitz equivalent, so we also have $c_k \leq \frac{b'}{l_{g_1}(\hat\alpha_k)}$. As $L_{g_0}$ assigns zero measure to geodesics which are tangent to a wall (Lemma \ref{lem:walls zero}), we must have that $\rho_k\to 0$ as $k\to \infty$. Therefore, with equation \eqref{eqn:length comp} and using Proposition \ref{prop:cty},
\[ \hat \imath(L_{g_0}, L_{g_0}) \leq \hat \imath(L_{g_0}, L_{g_1}). \]
Letting $c_k\langle \alpha_k \rangle \to L_{g_1}$ instead and using the same argument as well as the symmetry of $\hat \imath(-,-)$ gives
\[ \hat \imath(L_{g_0}, L_{g_1}) \leq \hat \imath(L_{g_1}, L_{g_1}). \]
Then we have, using Proposition \ref{prop:hat int vol},
\[ 4\pi Vol_{g_0}(X) = \hat \imath(L_{g_0}, L_{g_0}) \leq \hat \imath(L_{g_0}, L_{g_1}) \leq \hat \imath(L_{g_1},L_{g_1}) = 4\pi Vol_{g_1}(X). \]

\end{proof}

\begin{rem}
The proof of Theorem \ref{thm:MLS-determines-volume} will not work for the metric-dependent intersection pairings $i_{g_0}(-,-)$ and $i_{g_1}(-,-)$. If we attempt the argument above, in equation \eqref{eqn:length comp} we must use $i_{g_0}(-,-)$ on the left-hand side and $i_{g_1}(-,-)$ on the right-hand side. We obtain $i_{g_0}(L_{g_0},L_{g_0}) \leq i_{g_1}(L_{g_0},L_{g_1})$. Our second application of this argument proves $i_{g_0}(L_{g_0},L_{g_1})\leq i_{g_1}(L_{g_1},L_{g_1})$. These inequalities no longer patch together as desired.
\end{rem}

\bibliographystyle{alpha}
\bibliography{biblio}

\begin{thebibliography}{DMV11}

\bibitem[AB87]{alperin-bass}
R.~Aplerin and H.~Bass.
\newblock Length functions of group actions on $\lambda$-trees.
\newblock In {\em Combinatorial group theory and topology (Alta, Utah, 1984)},
  volume 111 of {\em Ann. of Math. Stud.}, pages 265--378, Princeton, NJ, 1987.
  Princeton Univ. Press.

\bibitem[BB95]{ball_brin}
Werner Ballmann and Michael Brin.
\newblock Orbihedra of nonpositive curvature.
\newblock {\em Publications Math\'ematiques de l'{I.H.\'E.S.}}, 82:169--209,
  1995.

\bibitem[Bil68]{billingsley}
Patrick Billingsley.
\newblock {\em Convergence of probability measures}.
\newblock John Wiley \& Sons, 1968.

\bibitem[BL17]{bankovic-leininger}
A.~Bankovi\'c and C.~Leininger.
\newblock Marked length spectral rigidity for flat metrics.
\newblock {\it Trans. Amer. Math. Soc.} (published online at
  https://doi.org/10.1090/tran/7005), 2017.

\bibitem[Bon88]{bonahon_geometry}
Francis Bonahon.
\newblock The geometry of {T}eichm\"uller space via geodesic currents.
\newblock {\em Inventiones Mathematicae}, 92:139--162, 1988.

\bibitem[Bon91]{bonahon_negative}
Francis Bonahon.
\newblock Geodesic currents on negatively curved groups.
\newblock In Roger~C. Alperin, editor, {\em Arboreal Group Theory}, volume~19
  of {\em Mathematical Sciences Research Institute Publications}, pages
  143--168. Springer, 1991.

\bibitem[Bou96]{Bou1}
M.~Bourdon.
\newblock Sur le birapport au bord des {CAT}(-1)-espaces.
\newblock {\em Publications Math\'ematiques de l'{I.H.\'E.S.}}, 83:95--104,
  1996.

\bibitem[Bou97]{Bou2}
M.~Bourdon.
\newblock Immeubles hyperboliques, dimension conforme, et rigidit\'e de
  {M}ostow.
\newblock {\em Geometric and Functional Analysis}, 7:245--268, 1997.

\bibitem[Bou00]{Bou3}
M.~Bourdon.
\newblock Sur les immeubles {F}uchsiens et leur type de quasi-isom\'etrie.
\newblock {\em Ergodic Theory and Dynamical Systems}, 20:343--364, 2000.

\bibitem[BP00]{BP}
M.~Bourdon and H.~Pajot.
\newblock Rigidity of quasi-isometries for some hyperbolic buildings.
\newblock {\em Comment. Math. Helv.}, 75:701--736, 2000.

\bibitem[Bro89]{brown}
Kenneth~S. Brown.
\newblock {\em Buildings}.
\newblock Springer-Verlag, New York-Berlin, 1989.

\bibitem[CD04]{croke-dairbekov}
Christopher~B. Croke and Nurlan~S. Dairbekov.
\newblock Lengths and volumes in {R}iemannian manifolds.
\newblock {\em Duke Mathematical Journal}, 125(1):1--14, 2004.

\bibitem[CFS82]{CFS}
I.~P. Cornfeld, S.~V. Fomin, and Ya.~G. Sinai.
\newblock {\em Ergodic Theory}.
\newblock Springer-Verlag, 1982.

\bibitem[CL]{CL}
D.~Constantine and J.-F. Lafont.
\newblock Marked length rigidity for one dimensional spaces.
\newblock (to appear in {\it J. Topol. Anal.}; available at arXiv:1209.3709).

\bibitem[CM87]{culler-morgan}
M.~Culler and J.~W. Morgan.
\newblock Group actions on $\mathbb{R}$-trees.
\newblock {\em Proc. London Math. Soc.}, 55:571--604, 1987.

\bibitem[Con17]{constantine}
D.~Constantine.
\newblock Marked length spectrum rigidity in nonpositive curvature with
  singularities.
\newblock 2017.
\newblock (to appear in {\it Indiana Univ. Math. J.}; available at
  https://www.iumj.indiana.edu/IUMJ/Preprints/7545.pdf).

\bibitem[Cro90]{croke}
C.~Croke.
\newblock Rigidity for surfaces of nonpositive curvature.
\newblock {\em Comment. Math. Helv.}, 65(1):150--169, 1990.

\bibitem[DK02]{dalbo-kim}
F.~Dal'Bo and I.~Kim.
\newblock Marked length rigidity for symmetric spaces.
\newblock {\em Comment. Math. Helv.}, 77:399--407, 2002.

\bibitem[DMV11]{DMV}
G.~Daskalopoulos, C.~Mese, and A.~Vdovina.
\newblock Superrigidity of hyperbolic buildings.
\newblock {\em Geometric and Functional Analysis}, 21:905--919, 2011.

\bibitem[Fan04]{Fa}
H.-R. Fana\"i.
\newblock Comparaison des volumes des vari\'et\'es {R}iemanniennes.
\newblock {\em C. R. Math. Acad. Sci. Paris}, 339:199--201, 2004.

\bibitem[FH64]{FH}
W.~Feit and G.~Higman.
\newblock The nonexistence of certain generalized polygons.
\newblock {\em J. Algebra}, 1:114--131, 1964.

\bibitem[Ham90]{ham}
Ursula Hamenst\"adt.
\newblock Entropy-rigidity of locally symmetric spaces of negative curvature.
\newblock {\em Annals of Mathematics}, 131(2):35--51, 1990.

\bibitem[HP97]{hersonsky-paulin}
S.~Hersonsky and F.~Paulin.
\newblock On the rigidity of discrete isometry groups of negatively curved
  spaces.
\newblock {\em Comment. Math. Helv.}, 72:349--388, 1997.

\bibitem[LL10]{LL}
F.~Ledrappier and S.-H. Lim.
\newblock Volume entropy for hyperbolic buildings.
\newblock {\em Journal of Modern Dynamics}, 4:139--165, 2010.

\bibitem[Ota90]{otal}
J.-P. Otal.
\newblock Le spectre marqu\'e des longueurs des surfaces \`a courbure
  n\'egative.
\newblock {\em Annals of Mathematics}, 131(1):151--160, 1990.

\bibitem[Ota92]{otal-2}
J.-P. Otal.
\newblock Sur la g\'eom\'etrie symplectique de l'espace des g\'eod\'esiques
  d'une vari\'et\'e \`a courbure n\'egative.
\newblock {\em Revista Matematica Iberoamericana}, 8:441--456, 1992.

\bibitem[PT09]{payne-thas}
S.~E. Payne and J.~A. Thas.
\newblock {\em Finite generalized quadrangles}.
\newblock European Mathematical Society, 2009.

\bibitem[San04]{santalo}
Luis~A. Santal\'o.
\newblock {\em Integral geometry and geometric probability}.
\newblock Cambridge University Press, 2004.

\bibitem[Sun15]{sun}
Z.~Sun.
\newblock Marked length spectra and areas of non-positively curved cone
  metrics.
\newblock {\em Geometriae Dedicata}, 178:189--194, 2015.

\bibitem[vM98]{van-Maldeghem}
H.~van Maldeghem.
\newblock {\em Generalized polygons}.
\newblock Birkh\"auser Verlag, 1998.

\bibitem[Wei03]{Weiss-structure}
R.~M. Weiss.
\newblock {\em The structure of spherical buildings}.
\newblock Princeton Univ. Press, 2003.

\bibitem[Xie06]{X}
Xiangdong Xie.
\newblock Quasi-isometric rigidity of {F}uchsian buildings.
\newblock {\em Topology}, 45:101--169, 2006.

\end{thebibliography}

\end{document}